\documentclass[10pt]{amsart}

\usepackage{palatino, mathpazo}
\usepackage[mathscr]{eucal}
\usepackage{amsmath,amsthm,amsfonts,amssymb,mathrsfs,textcomp,yfonts}
\usepackage[all]{xy}
\usepackage{hyperref}

\newtheorem{theorem}{Theorem}[section]
\newtheorem{lemma}{Lemma}[section]

\theoremstyle{remark}
\newtheorem{remark}[theorem]{Remark}

\newtheorem{definition}{Definition}[section]
\numberwithin{equation}{section}

\newcommand{\I}{\mathscr{I}}
\newcommand{\X}{\mathfrak{X}}
\newcommand{\la}{\langle}
\newcommand{\laa}{\langle\!\langle}
\newcommand{\ra}{\rangle}
\newcommand{\raa}{\rangle\!\rangle}

\title[Extensions of multiply twisted pluri-canonical forms]
{Extensions of multiply twisted\\ pluri-canonical forms}

\author{Chen-Yu Chi}
\address{C.-Y. Chi: Department of Mathematics, Harvard University,
Cambridge}
\email{cychi@math.harvard.edu}

\author{Chin-Lung Wang}
\address{C.-L. Wang: Department of Mathematics, National Taiwan
University, Taipei}
\email{dragon@math.ntu.edu.tw}

\author{Sz-Sheng Wang}
\address{S.-S. Wang: Department of Mathematics, National Taiwan
University, Taipei}
\email{d98221004@math.ntu.edu.tw}

\begin{document}
\everymath{\displaystyle}

\maketitle

\section{Introduction}
In this work we study the problem of extending ``multiply twisted"
pluri-canonical forms from smooth divisors in a complex projective
manifold. We first state the main theorem and then review some
earlier results. Definitions and notation can be found in Section
\ref{defcon}.

\begin{theorem}\label{mainthm}
Let $X$ be a projective manifold of dimension $n$, $D\subset X$ a
smooth divisor with canonical section $s_D$.

Let $h_D$ be an almost semipositive metric (cf.\ \ref{semip}) on
the line bundle $D$ such that $|s_D|_{h_D}$ is essentially bounded
on $X$, i.e.~bounded by a fixed number almost everywhere, and let
$(L_1,h_1),\dots,(L_m,h_m)$ be semipositive line bundles (cf.\
\ref{semip}) such that the restriction of the singular metric
$h_j$ to $L_j|_D$ is well defined, i.e.~not identically $+\infty$
along $D$.

If there is a real number $\mu > 0$ such that
$$
\mu \sqrt{-1} \Theta_{h_j}\geqslant \sqrt{-1} \Theta_{h_D}
$$
as currents on $X$ for $j=1,\dots,m$, then for every section $\sigma$ of
$$
\bigotimes_{j = 1}^m (K_D + L_j|_D)\otimes \I_1\I_2\cdots\I_m
$$
on $D$, where $\I_j$ denote the multiplier ideal sheaves $\I(
h_j|_D)$, there exists a global section $\widetilde{\sigma}$ of
$$
\bigotimes_{j = 1}^m (K_X +D + L_j) = m(K_X+D) + L_1 + \cdots + L_m
$$
on $X$ such that $\widetilde{\sigma}|_{D} = \sigma \wedge
(ds_D)^{\otimes m}$ (cf.\ $\ref{adj}$).
\end{theorem}

Extension theorems of this type (for $m = 1$) date back to the
work of Ohsawa and Takegoshi \cite{ot87} on extending holomorphic
functions from submanifolds of Stein manifolds with weighted $L^2$
estimates. Their key idea is to use a modified Bochner--Kodaira
inequality to achieve the $L^2$ estimate for a skewed
$\overline\partial$ operator. This theorem was generalized by
Manivel \cite{man} to the case of holomorphic sections of vector
bundles. Variants of their theorems were used by Angehrn and Siu
\cite{a-s}, in their study of Fujita's conjecture, to prove the
semicontinuity of multiplier ideal sheaves under variation of the
singular metrics, and used by Siu \cite{siu1, siu2}, in his proof
of the invariance of plurigenera, to extend pluricanonical forms
from the central fiber of a smooth projective family of complex
manifolds to the total space.

The argument exploited in \cite{siu2} was generally referred to as
a ``two tower" argument by Siu. Indeed, in \cite{siu2}, the
theorem of Ohsawa-Takegoshi type ($m  = 1$) is for the canonical
bundle twisted by a suitable line bundle. In passing from a single
canonical bundle to pluricanonical bundles, Siu combined the
extension theorem with Skoda's theorem on (effective) ideal
generation as well as a supremum norm estimate. Later
P$\breve{\mathrm a}$un \cite{pau} simplified Siu's approach by
showing that the supremum norm condition can be replaced by an
$L^2$ one and the invariance of plurigenera can be deduced
directly from the extension result without using Skoda's theorem.
More precisely, he proved the following result:

\begin{theorem} [P$\breve{\mathrm a}$un \cite{pau}] \label{pau}
Let $\pi : \X \rightarrow \Delta$ be a projective family over the
unit disk and $(L,h)$ a semipositive line bundle on $\X$ such that
the restriction $h|_{\X_0}$ is well defined. Then every section of
$(mK_{\X_0} + L|_{\X_0}) \otimes \mathscr{I}(h|_{\X_0})$ on $\X_0$
extends to a section of $mK_\X+L$.
\end{theorem}

His proof consists of an elegant single tower climbing induction
argument. The induction is on the multiple of the canonical bundle
twisted by the {\it fixed} line bundle $L$ equipped with a {\it
fixed} singular metric $h$. It is then natural to ask, when
climbing the tower, can we add different line bundles each with
its own singular metric instead of just a constant pair $(L,h)$.
If this can be achieved, one may possibly obtain an extension
theorem of ``multiply twisted" pluricanonical forms. In fact,
Demailly proved the following result:

\begin{theorem} [Demailly \cite{dema2}] \label{dema2}
Let $\X$ and $\pi$ be as in Theorem $\ref{pau}$ and
$(L_j,h_j)$ $(1 \leqslant j \leqslant m)$ semipositive line
bundles on $\X$ such that $h_j |_{\X_0}$ are well defined.
Suppose $\mathscr{I}(h_j|_{\X_0})=\mathscr{O}_{\X_0}$ for
$j=2,\dots,m$. Then every section of
$(mK_{\X_0}+L_1|_{\X_0}+\cdots+L_m|_{\X_0})
\otimes\mathscr{I}(h_1|_{\X_0})$ on $\X_0$ extends to a
section of  $(mK_\X+L_1+\cdots+L_m)$.
\end{theorem}

Note that, although Theorem \ref{dema2} enables one to add
different line bundles $L_j$, only one of them is allowed to be
equipped with a singular metric whose multiplier ideal sheaf is
nontrivial. This motivates us to look at the statement like
Theorem \ref{mainthm}, which removes this restriction. This was
recently achieved in \cite{s.wang}.

\begin{theorem} [\cite{s.wang}] \label{s.wang}
Let $\pi : \X \rightarrow \Delta$ be a projective family over the
unit disk and $(L_j,h_j)$ $(1 \leqslant j \leqslant m)$
semipositive line bundles on $\X$ such that $h_j|_{\X_0}$ are well
defined. Then every section of
$(mK_{\X_0}+L_1|_{\X_0}+\cdots+L_m|_{\X_0})\otimes \I_1\I_2 \cdots
\I_m$ on $\X_0$ extends to a section of $(mK_\X+L_1+\cdots+L_m)$
on $\X$, where $\I_j$ is the multiplier ideal sheave $\I(
h_j|_{\X_0})$ on $\X_0$.
\end{theorem}

Inspired by the results of Tsuji, Takayama, and Hacon-McKernan
respectively in connection with their work on pluricanonical
series \cite{tsu}, \cite{tak}, and \cite{hk}, we proved our
theorem under the setting of pairs of a complex projective
manifold and a smooth divisor whose associated line bundle
satisfies some conditions on curvature. The projective family case
is relatively easier in that the line bundle associated to the
central fiber is trivial, hence it can be ignored in the necessary
curvature condition, i.e.~the curvature inequality in Theorem
\ref{mainthm} holds automatically.

Most of our arguments in the proof of Theorem \ref{mainthm} follow
closely P$\breve{\mathrm{a}}$un's one tower argument. The major new
input to overcome the non-triviality of multiplier ideal sheaves
$\mathscr{I}( h_j |_D)$, which occurs during the intermediate
inductive steps, is a more careful choice of the auxiliary twisting
ample line bundle (denoted by $A$ in our argument). This bundle
needs to be sufficiently ample to take care of both the required
metric properties and the global generation for related coherent
sheaves. The complete discussion is presented in Section
\ref{main.steps} and Section \ref{h_infty}.

For completeness and self-containedness of this article, we include
in Appendix 1 (Section \ref{app}) a proof of the Ohsawa--Takegoshi
type theorem which we will use. The proof is exactly the same as
the proof in \cite{siu2}, except that we deal with the situation in
which the line bundle $D$ is not trivial. 
A similar statement appeared in \cite{va}, Theorem 2. 
It is worth noting that
Friedrichs and H\"ormander's results (\cite{fr} and \cite{hor1}) on
the density in the graph norm (cf.\  Remark \ref{dens gr}) plays an
essential role when using the Bochner--Kodaira formula to get a
priori estimates. This density result requires the weight functions
to be smooth or to have at most suitably mild singularities. 
Therefore, to allow $h_D$ to be
a singular metric, one has to reduce the proof to the case when it
is smooth. 
We discuss
such a reduction in detail for completeness, although it might be
well known to experts. In addition, Theorem \ref{mainthm} is a
refinement of \cite{va}, Theorem 1.

In fact we only dealt with the case $h_D$ being smooth in our first
version submitted on October 2010 since we were still struggling on
this subtle regularization issue at that time. We developed our
treatment in Appendix 1 following ideas of Siu which we learnt from
several of his lectures and private notes. We consider a locally
biholomorphic projection from a Stein manifold to a Euclidean space
and apply the convolution method on the target Euclidean space.

We also noticed that in a recent preprint by Demailly, Hacon, and
P$\breve{\mathrm{a}}$un \cite{dhp}, an extension theorem similar to
Theorem \ref{otsa u} has been proven. They also gave a detailed
discussion on the process of smoothing singular metrics. Their
approach is basically as follows. First one imbeds a Stein manifold
$V$ (which will be the complement of some suitable sufficiently
ample divisor $H$ in the projective manifold $X$ under
consideration) in an ambient $M$ (which is an Euclidean space in
their case). Then, by a theorem of Siu (Theorem 4.2 in \cite{dhp})
one can construct a Stein neighborhood $W$ of $V$ in the ambient
space $M$ which admits a holomorphic retraction $r:W\to V$. To
smoothen plurisubharmonic functions on $V$, one first pulls them
back to $W$ via $r$, which are still plurisubharmonic. After
applying the usual convolution method in the Euclidean space $M$ to
regularize the pulled back functions, one takes their restrictions
on $V$.

These two methods are different. Although both methods crucially
use the Stein property and convolution, the difference lies in
that the approach in \cite{dhp} is ``injective" and ours is
``projective".

We are able to extend Theorem \ref{mainthm} to allow $L_j$'s to be
${\bf R}$ divisors instead of genuine line bundles. We are
grateful to the referee for asking this question. Since the proof
requires some other techniques, we will present it in a separate
work. \\

{\bf Acknowledgements.} This collaboration arose from discussions
during the seminar series ``Analytic Approach to Algebraic
Geometry'' in December 2008 and March 2010 at National Taiwan
University sponsored by the National Center for Theoretic Sciences
and Taida Institute of Mathematical Sciences. Two major references
we studied are Siu's Harvard lecture notes on ``Complex geometry''
and the excellent online book ``Complex analytic and differential
geometry'' written by Demailly. We are grateful to both authors
for their inspiring writings and generous sharing. Also we would
like to thank the referee for pointing out a gap in an earlier
version of the proof of Lemma \ref{peslem}, which led us to
formulate the almost semipositivity condition for $h_D$ in our
Theorem \ref{mainthm}.

C.-L.\ would like to express his sincere gratitude to Professor Eckart Viehweg for providing crucial help during his early stage of mathematical career. 

\section{{Preliminaries and Conventions}}\label{defcon}
\subsection{{Adjunction}}\label{adj}
Given a smooth divisor $D$ in a compact complex manifold $X$, we
use the same letter $D$ to denote the line bundle associated to
$D$. In order to justify the restriction of sections of adjoint
line bundles on $X$ to get sections of adjoint line bundles on
$D$, we need to take a closer look at the adjunction formula
$K_D\simeq (K_X+D)|_D$. Locally $D$ is given by a set of equations
$\{s_\alpha=0\}$ with respect to an open cover $\{U_\alpha\}$. The
relations $s_\alpha=g_{\alpha\beta}s_\beta$ on $U_\alpha\cap
U_\beta$ give a $1$-cocycle $\{g_{\alpha\beta}\}$ of the sheave
$\mathscr O^*_X$ which defines the line bundle $D$, and
tautologically the locally defined functions $s_\alpha$'s give a
canonical section, denoted by $s_D$, which is unique up to scaling
and will be fixed throughout all arguments. The short exact
sequence
$$
0\rightarrow N_{D/X}^*\rightarrow
T_X^*|_D\rightarrow T_D^*\rightarrow 0
$$
implies a canonical isomorphism
by taking wedge product:
$$
K_D+ N_{D/X}^*=K_X|_D.
$$
(We adopt the additive notation for tensor products of line bundles.)

On the other hand, $ds_\alpha$ is a local frame of $N_{D/X}^*$ on
$U_\alpha$. Let $e_{\alpha}$ be a local frame of $D$ on $U_\alpha$
for all $\alpha$. The relation $s_\alpha=g_{\alpha\beta}s_\beta$
and $e_\beta=g_{\alpha\beta}e_\alpha$ implies that
$\{ds_\alpha\otimes e_\alpha\}$ defines a global frame, denoted by
$ds_D$, of the line bundle $N_{D/X}^*+ D|_D$, and hence
$N_{D/X}^*+ D|_D$ is trivial. This induced the isomorphism
$$K_D\simeq K_D+ N_{D/X}^*+ D|_D=K_X|_D+ D|_D$$ by sending $\eta$ to
$\eta\wedge ds_D$.

\subsection{{Singular metrics and pseudonorms}}\label{sm}

The term ``singular hermitian metric" or ``singular metric" always
means a hermitian metric whose local weight functions are locally
Lebesgue integrable, and hence smooth metrics are counted as
singular metrics. For such metrics $h$ we use $\Theta_h$ to denote
their curvature currents. Locally we have $h = e^{-\varphi}$ with
$\Theta_{h} = -\partial\bar\partial \log e^{-\varphi} =
\partial\bar\partial \varphi$.

Let $X$ be a complex manifold of dimension $n$ and $L$ a line
bundle on $X$ with a singular metric $h$. Let $s$ be a (Lebesgue)
measurable section of $mK_X+L$. Suppose $s$ and $h$ are
represented by functions $f(z)$ and $h(z)$ in terms of local
coordinates $z=(z^1,\dots,z^n)$, $z^j=x^j+{\sqrt{-1}}y^j$, of
trivializing charts of $L$.

\begin{definition}\label{2/m}
We define a measurable $(n,n)$-form $\la s\ra_{h}^{\frac{2}{m}}$ by
setting
$$
\la s\ra_{h}^{\frac{2}{m}}=h(z)^\frac{1}{m}|f(z)|^\frac{2}{m}dx^1\wedge
dy^1\wedge\cdots\wedge dx^n\wedge dy^n
$$
locally. $\la
s\ra_{h}^{\frac{2}{m}}$ is clearly well defined and is nonnegative with
respect to the canonical orientation on $X$ associated to $dx^1\wedge
dy^1\wedge\cdots\wedge dx^n\wedge dy^n$. Therefore we define
$$
\laa s\raa_{h}=\int_X\la s\ra_{h}^{\frac{2}{m}}\leqslant\infty.
$$
This number is called the pseudonorm of $s$ with respect to $h$.
\end{definition}

Suppose $g$ is a smooth hermitian metric on $T_X$ with K\"ahler
form $\omega$. $g$ induces a hermitian metric on the canonical
bundle $K_X$, denoted as $g_\omega$. Let
$dV_\omega=\frac{\omega^n}{n!}$ be the volume form on $X$ induced
by $g$. It is easily seen that
$$
\la s\ra_{h}^{\frac{2}{m}}=|s|_{g_\omega^{\otimes m}\otimes
h}^{\frac{2}{m}}dV_\omega.
$$
Using this expression one sees directly the following
facts:\smallskip

(i) Suppose $L$ and $L'$ are two line bundles with singular
metrics $h$ and $h'$ respectively. For any measurable sections $s$
of $mK_X+L$ and $s'$ of $L'$, and $l\in\mathbf N$ we have
\begin{equation}\label{||<>}
\la s\otimes s'\ra_{h\otimes
h'}^{\frac{2}{m}}=|s'|_{h'}^{\frac{2}{m}}\la
s\ra_{h}^{\frac{2}{m}}
\end{equation}
and
\begin{equation}\label{homo}
\la s^l\ra^{\frac{2}{lm}}_{ h^{\otimes l}}=\la s\ra^{\frac{2}{m}}_{h}.
\end{equation}

(ii) If $s_j$ is a measurable section of $m_jK_X+L_j$ and $h_j$ is
a singular metric on $L_j$, $j=1,\dots,r$, then we can deduce from
the usual H\"older inequality the ``H\"older inequality for
pseudonorms":
\begin{equation}\label{Holder}
\laa s_1\otimes\cdots\otimes s_r\raa_{ h_1\otimes\cdots\otimes
h_r}^{m_1+\cdots+m_r}\leqslant\laa s_1\raa_{
h_1}^{m_1}\cdots\laa s_r\raa_{ h_r}^{m_r}.
\end{equation}

\subsection
{{Almost semipositive line bundles and pseudoeffective divisors}}
\label{semip} A {\it semipositive} line bundle (resp.\ an almost
semipositive line bundle) is a pair $(L, h)$ of a line bundle $L$
and a singular hermitian metric $h$ on $L$ such that
${\sqrt{-1}}\Theta_{h}$ is a {\it closed positive current} in the
sense of Lelong (resp.\ the sum of a closed positive current and a
smooth $(1,1)$-form), or equivalently, each of its local weights
is a nontrivial plurisubharmonic function, i.e.~not identically
$-\infty$ (resp.\ the sum of a nontrivial plurisubharmonic
function and a smooth function). We will call such $h$ a
semipositive metric (resp. an almost semipositive metric) on $L$.
The multiplier ideal sheaf associated to an almost semipositive
singular metric $h$ is the coherent sheaf of local $L^2_h$
sections and is denoted by $\I_{h}$ or by $\I(h)$.

\begin{remark} \label{almost semip}
On a projective manifold $X$, a pair $(L,h)$ is almost
semipositive if and only if there exist a semipositive line bundle
$(L_1,h_1)$ and a line bundle with smooth hermitian metric
$(L_2,h_2)$ such that $L=L_1\otimes L_2$ and $h=h_1\otimes h_2$.
\end{remark}

A typical type of semipositive line bundles consists of effective
line bundles by the following construction.

\begin{definition}\label{semipo}
Let $S=\{s_1,\dots,s_l\}$ be a set of nontrivial global
holomorphic sections of a line bundle $L$. For any $\sigma\in L_x$
where $x\in X$, we choose an arbitrary smooth metric $h$ on $L$
and define
$$|\sigma |_{{h}_S}^2:=\frac{|\sigma |_h^2}{\sum\limits_{j=1}^l
|s_j(x)|_h^2}.$$
\end{definition}
If $s$ is a section of $K_X+L$ and $S=\{s_1,\dots,s_l\}$ a set of
global holomorphic section of $L$, then for any smooth metric $h$
on $L$ we have
\begin{equation}\label{pdo}
\la s\ra^2_{ h_S}=\frac{\la s\ra_h^2}{\sum\limits_{j=1}^l |s_j|_h^2}.
\end{equation}
It is clear that the definition does not depend on the choice of
$h$. Locally if the sections $\{s_j\}$ are represented by
functions $\{f_j\}$ then the weight function is
$$
\varphi:=\log\Big(\sum_{j=1}^l|f_j|^2\Big)
$$
which is plurisubharmonic, and hence $\sqrt{-1}\Theta_{ h_S} =
{\sqrt{-1}} \partial\overline\partial \log (\Sigma_j|f_j|^2)
\geqslant 0$.

Denote by $\mathrm{Psef}(X) \subseteq N^1(X)_{\bf R}$ the closure
of the real convex cone generated by numerical classes of
semipositive line bundles over $X$. In the algebraic case, we have
the following interpretation.

\begin{remark} (cf.\ \cite{dema})
If $X$ is projective then $\mathrm{Psef}(X) =
\overline{\mathrm{Eff}(X)} = \overline{\mathrm{Big}(X)}$, where
$\overline{\mathrm{Eff}(X)}$ (resp.\ $\overline{\mathrm{Big}(X)}$)
is the closure of effective (resp.\ big) cone of $X$, which is
also known as the cone of pseudoeffective divisors.
\end{remark}

\section{The main extension result} \label{main.steps}
\subsection{An extension theorem for adjoint line bundles}
We will need the following extension theorem of Ohsawa--Takegoshi
type for adjoint line bundles, whose proof will be given in
Appendix 1.

\begin{theorem} \label{otsa u}
Let $X$ be a projective manifold, $D \subseteq X$ a smooth
divisor. Suppose $h_D$ an almost semipositive metric on the line
bundle $D$ such that $|s_D|_{h_D}$ is essentially bounded on $X$
and $(L,  h)$ be a semipositive line bundle on $X$. If there is a
real number $\mu > 0$ such that
$$
\mu\sqrt{-1}\Theta_{h}\geqslant \sqrt{-1} \Theta_{h_D}
$$
as currents on $X$, then for every section $s$ of $(K_D +
L|_D)\otimes\I(h|_D)$ there exists a section $\widetilde{s}$ of
$K_X + D + L$ such that $\widetilde{s}|_D = s\wedge ds_D$ and
$$
\int_X \la\widetilde{s}\ra_{h_D \otimes  h}^2 \leqslant C \int_D \la
s\ra_{h}^2
$$
where $C>0$ only depends on ${{\rm ess.}\sup}_X |s_D|_{h_D}$ and $\mu$.
\end{theorem}

Note that the statement of Theorem \ref{mainthm} for $m = 1$ is
exactly the statement of Theorem \ref{otsa u}. Hence we fix from
now on a positive integer $m \ge 2$ and consider a non-zero
$\sigma$ as in the hypothesis of Theorem $\ref{mainthm}$.

\subsection{Reduction to constructing a semipositive
metric on $m(K_X + D)+\sum\nolimits_{1}^m L_j$} \label{amp}

Note that $m(K_X+D)+\sum\nolimits_{1}^m L_j=K_X+D+(m-1)(K_X+D)
+\sum\nolimits_{1}^m L_j$. In order
to prove Theorem \ref{mainthm} via Theorem \ref{otsa u}, we
need to create a semipositive metric $ h_0$ on
$(m-1)(K_X+D)+\sum\nolimits_{1}^m L_j$ such that
$$
\mu\sqrt{-1}\Theta_{ h_0}\geqslant\sqrt{-1}\Theta_{h_D}$$
as currents and
$$
\int_D\la\sigma\wedge ds_D^{\otimes(m-1)}\ra_{ h_0}^2<\infty.
$$
The construction of $ h_0$ goes as follows. First, we choose $A$
to be so ample that the following conditions hold: \\

$(A_1)$ For each $r=0, 1, \ldots, m-1$, the line bundle $(m-r)A$
is generated by its global sections $\{t_{l}^{(r)}\}_{1 \leqslant
l \leqslant
N}$.\\

$(A_2)$ The coherent sheaf $(K_D + L_j|_D + A|_D) \otimes \I_j$ on
$D$ is generated by its global sections $\{ s_{j, l} \}_{1
\leqslant l \leqslant N}$ for each $1 \leqslant j \leqslant m$.\\

$(A_3)$ The following map induced by $\I_1 \otimes \cdots \otimes
\I_m \to \I_1 \cdots \I_m$ is surjective:
\begin{align*}
&\bigotimes\limits_{j = 1}^m H^0\left(D, (K_D + L_j|_D + A|_D) \otimes
\I_j\right) \\
&\qquad \longrightarrow H^0\left(D, (mK_D + \sum\nolimits_1^m
L_j|_D + mA|_D) \otimes \I_1 \cdots \I_m\right).
\end{align*}
This can be achieved by Lemma $\ref{s.lem}$ in Appendix 2.\\

$(A_4)$ Every section of $\big(m (K_X +D) + \sum\nolimits_{1}^m
L_j + mA\big)|_D$ on $D$ extends to $X$. This is a consequence of
the Serre vanishing theorem.\\

Suppose that we have a semipositive metric $ h_\infty$ (which will
be constructed in Lemma \ref{hinfty} by using the auxiliary ample
bundle $A$) on $m(K_X+D)+\sum\nolimits_{1}^m L_j$ such that
$\big|\sigma\wedge ds_D^{\otimes m}\big|_{ h_\infty}\leqslant 1$.
We take $ h_0= h_\infty^{\frac{m-1}{m}}( h_1\cdots
h_m)^{\frac{1}{m}}$.
The curvature condition holds since
$$
\mu\sqrt{-1}\Theta_{ h_0}=\frac{\mu(m-1)}{m}\sqrt{-1}\Theta_{
h_\infty}+\frac{1}{m}\sum_{j=1}^m\mu\sqrt{-1}\Theta_{
h_j}\geqslant\sqrt{-1}\Theta_{h_D}
$$
by the curvature assumption in Theorem \ref{mainthm}.

The
finiteness condition also holds. To see this, first note that, by
(\ref{||<>}) and (\ref{homo}),
\begin{align*}\label{a}
&\la\sigma\wedge ds_D^{\otimes(m-1)}\ra_{h_0}^2=\la(\sigma\wedge
ds_D^{\otimes (m-1)})^{\otimes m}\ra_{h_0^{\otimes m}}^{\frac{2}{m}}\\
&\quad =\la(\sigma\wedge ds_D^{\otimes m})^{\otimes(m-1)}\otimes\sigma
\ra_{h_\infty^{\otimes(m-1)} \otimes h_1 \otimes \cdots
\otimes h_m}^{\frac{2}{m}}\\
&\qquad =\big|(\sigma\wedge ds_D^{\otimes m})^{\otimes(m-1)}
\big|_{h_\infty^{\otimes(m-1)}}
^{\frac{2}{m}}\la\sigma\ra_{h_1\otimes\cdots\otimes h_m}^{\frac{2}{m}}\\
&\quad \qquad =\left(\big|\sigma\wedge ds_D^{\otimes
m}\big|_{h_\infty}^2 \right)
^{\frac{m-1}{m}}\la\sigma\ra_{h_1\otimes\cdots\otimes
h_m}^{\frac{2}{m}} \leqslant\la\sigma\ra_{h_1\otimes\cdots\otimes
h_m}^{\frac{2}{m}}.
\end{align*}
By $(A_3)$,
$$
\sigma\otimes
t^{(0)}_l=\sum_{p=1}^{n_l}\tau_{l;1,p}\otimes\cdots\otimes\tau_{l;m,p}
$$
where $\tau_{l;j,p}$ are sections of $(K_D+L_j|_D+A|_D)\otimes\I_{
h_j|D}$ for $l=1,\dots,N$. Again, by ($\ref{||<>}$) and
($\ref{homo}$),
\begin{align*}
&\Big(\sum\limits_{l=1}^N\big|t^{(0)}_l \big|_{h_A^{\otimes
m}}^{\frac{2}{m}} \Big)\la\sigma\ra_{h_1\otimes\cdots\otimes
h_m}^{\frac{2}{m}} = \sum\limits_{l=1}^N\la\sigma\otimes
t^{(0)}_l\ra_{h_1\otimes\cdots\otimes
h_m\otimes h_A^{\otimes m}}^{\frac{2}{m}}\\
&\quad \leqslant \sum\limits_{l=1}^N\sum_{p=1}^{n_l}\la\tau_{l;1,p}
\otimes\cdots\otimes\tau_{l;m,p}
\ra_{h_1\otimes\cdots\otimes h_m\otimes h_A^{\otimes m}}^{\frac{2}{m}}
\end{align*}
where $h_A$ is a fixed smooth metric on $A$.
$$
M_0:=\min\limits_{D}\sum_l\big|t^{(0)}_l\big|_{h_A^{\otimes
m}}^{\frac{2}{m}} > 0
$$
exists since $\sum_l\big|t^{(0)}_l\big|_{h_A^{\otimes
m}}^{\frac{2}{m}}$ is a nonvanishing smooth function by $(A_1)$
and $D$ is compact. Therefore
$$
\la\sigma\ra_{h_1\otimes\cdots\otimes h_m}^{\frac{2}{m}}
\leqslant\frac{1}{M_0}\sum\limits_{l=1}^N\sum_{p=1}^{n_l} \la
\tau_{l;1,p}
\otimes\cdots\otimes\tau_{l;m,p}\ra_{h_1\otimes\cdots\otimes
h_m\otimes h_A^{\otimes m}}^{\frac{2}{m}}.
$$
By the above, $(A_1)$, and ($\ref{Holder}$),
\begin{align*}
&\int_D\la\sigma\wedge ds_D^{\otimes(m-1)}\ra_{ h_0}^2\\
&\quad
\leqslant\frac{1}{M_0}\sum\limits_{l=1}^N\sum_{p=1}^{n_l}\int_D
\la \tau_{l;1,p} \otimes \cdots \otimes
\tau_{l;m,p} \ra_{h_1 \otimes \cdots \otimes h_m \otimes
h_A^{\otimes m}}^{\frac{2}{m}}\\
&\qquad \leqslant\frac{1}{M_0} \sum \limits_{l=1}^N
\sum_{p=1}^{n_l}\bigg(\int_D \la\tau_{l;1,p}\ra_{h_1\otimes
h_A}^2\bigg)^{\frac{1}{m}} \cdots \bigg(\int_D \la\tau_{l;m,p}
\ra_{h_m\otimes h_A}^2\bigg)^{\frac{1}{m}} < \infty.
\end{align*}
 Applying Theorem
$\ref{otsa u}$ to prove Theorem $\ref{mainthm}$ is then justified
if such $h_\infty$ exists.\\

\section{Construction of the metric $h_\infty$} \label{h_infty}
\subsection{A modification of Siu and  P\u{a}un's induction}\label{aux}

Here we follow the argument in \cite{pau} and \cite{siu2}. For
every positive integer $k = qm + r$ ($q=[k/m]$ the Gauss symbol of
$k/m$ and $0\leqslant r \leqslant m-1$ the remainder), we let
$$
L^{(k)} := q \sum_{j = 1}^mL_j + L_1 + \cdots + L_r
$$
and let $F_k := k (K_X +D) + L^{(k)} + mA$ where $A$ is the ample
bundle chosen in \ref{amp}.

Were $m(K_X+D)+L^{(m)}$ known to have a family of sections which do
not vanish identically along $D$ and their restrictions to $D$ are
basically $\sigma\wedge ds_D^{\otimes(m)}$ multiplied by some
functions which do not have common zeros, we can simply take
$h_\infty$ to be the semipositive metric defined by them
(Definition $\ref{semipo}$).

However, we do not know a priori that $m(K_X+D)+L^{(m)}$ have any
nonzero sections (we are trying to produce one). Instead, for the
ample line bundle $A$ we can find a set of sections $S_k$ of
$F_k=k(K_X+D)+L^{(k)}+mA$ whose restrictions to $D$ have
properties similar to those mentioned above (Lemma $\ref{ind}$).
Then we try to obtain $h_\infty$ by ``taking the $q$-root" of the
semipositive metrics $h_{S_{qm}}$ on
$F_{qm}=q(mK_X+mD+L^{(m)})+mA$ to ``eliminate" the line bundle
factor $mA$ (Lemma $\ref{hinfty}$).

Now we let $\Lambda_r := \textstyle{\prod_1^r \{1, \ldots, N \}}$
for $r = 1, 2, \ldots, m - 1$.  For every $J=(j_1,\dots,j_r)\in
\Lambda_r$, we define
\begin{center}
$s^{(r)}_J := s_{1, j_1} \otimes \cdots \otimes\, s_{r, j_r}$
\end{center}
with the convention that $\Lambda_0=\{0\}$ $s_0^{(0)}:= 1$ for $r=
0$. We define the special index set $\Lambda_m^*$ to be
$\textstyle{\prod_1^m \{1, \ldots, N \}}$ and sections $\hat
s^{(m)}_{J}=s_{1,j_1}\otimes\cdots\otimes s_{m,j_m}$ for all
$J=(j_1,\dots, j_m)\in\Lambda_m^*$. We consider for each $k\geq m$
the
following statement:\\

$(E)_k$: There exists a family of sections
$$
S_k=\{\widetilde{\sigma}_{J,l}^{(k)}:J \in\Lambda_r, 1 \leqslant l
\leqslant N\}
$$
of $F_k$ over $X$ such that
\begin{equation}\label{sigma|D}
\widetilde{\sigma}_{J,l}^{(k)}|_D = \sigma^{\otimes[k/m]} \otimes
s^{(r)}_{J} \otimes t_{l}^{(r)}\wedge ds_D^{\otimes k}
\end{equation}
for all $J\in\Lambda_r$ and $l=1,\dots, N$, where $r=k-[k/m]m$.

\begin{lemma}\label{ind}
$(E)_k$ holds for all $k\geqslant m$. Moreover, there exists a
constant $C_0>0$ which only depends on ${{\rm ess.}\sup}_X
|s_D|_{h_D}$, $\mu$, $\sigma$, and the choices of
$\{t_{l}^{(r)}\}$ and $\{s_{j,l}\}$ in $(A_2)$ and $(A_3)$ above
such that
\begin{equation}\label{un}
\int_X\sum_{\substack{J\in\Lambda_r\\
l=1,\dots,N}}\la\widetilde\sigma_{J,l}^{(k)}\ra^2_{h_D\otimes
h_{S_{k-1}}\otimes h_{r^*}}\leqslant C_0
\end{equation}
for all $k>m$, where $r=k-[k/m]m$ and
$$
r^*:=\left\{
\begin{array}{ccc}
r&\text{ if }&r\neq 0,\\
m&\text{ if }&r=0.
\end{array}
\right.
$$
\end{lemma}

\begin{proof} First, $(E)_m$ holds by $(A_4)$. We proceed to prove that
$(E)_{k-1}$ implies $(E)_k$ for any $k>m$. Note that
$F_k=K_X+D+F_{k-1}+L_{r^*}$ and hence
$F_k|_D=K_D+(F_{k-1}+L_{r^*})|_D+(N_{D/X}^*+D|_D)$ by $\ref{adj}$.
We are going to apply Theorem $\ref{otsa u}$ to the situation
$L=F_{k-1}+L_{r^*}$ and $s=\sigma^{\otimes[k/m]}\otimes
s_J^{(r)}\otimes t_l^{(r)}\wedge ds_D^{\otimes (k-1)}$. We choose
the singular metric $h$ on $F_{k-1}+L_{r^*}$ to be $
h_{S_{k-1}}\otimes h_{r^*}$.

The restriction $ h_{S_{k-1}}|_D$ is well defined by
($\ref{sigma|D}$), $(A_1)$, $(A_2)$, and $(A_3)$; $ h_{r^*}|_D$ is
well defined by the hypothesis of Theorem $\ref{mainthm}$.
Therefore $h|_D$ is well defined. By $\ref{semip}$ and the
hypothesis of Theorem $\ref{mainthm}$,
$$
\mu\sqrt{-1}\Theta_{h}=\mu\sqrt{-1}\Theta_{
h_{S_{k-1}}}+\mu\sqrt{-1}\Theta_{ h_{r^*}}\geqslant\sqrt{-1}\Theta_{h_D}
$$
and the curvature condition is fulfilled.

In the following we will show that
$$
\int_D\la\sigma^{\otimes[k/m]}\otimes s_J^{(r)}\otimes t_l^{(r)}\wedge
ds_D^{\otimes (k-1)}\ra^2_{ h_{S_{k-1}}\otimes h_{r^*}}\leqslant C'
$$
for a positive number $C'$ which only depends on the choices of
$\{t_{l}^{(r)}\}$ and $\{s_{j,l}\}$ in $(A_2)$ and $(A_3)$ above.
This will imply $s$ is a section of
$\big(K_D+(F_{k-1}+L_{r^*})|_D\big)\otimes\I_{h}$ and, combined
with the pseudonorm inequality on Theorem $\ref{otsa u}$, will
yield ($\ref{un}$).
\\

{\bf Case 1}: $r\neq 0$, i.e.~$[k/m]=[(k-1)/m]$.\\

We choose smooth metrics $h_A$ on $A|_D$, $h^{(r-1)}$ on
$(r-1)K_D+L^{(r-1)}|_D$, and $h'$ on
$[k/m](mK_D+L^{(m)})+(k-1)(N_{D/X}^*+D|_D)$. We let $h:=h'\otimes
h^{(r-1)}\otimes h_A^{\otimes m}$ on $F_{k-1}|_D$. Writing
$J=(J_0',j_0)$ with $J_0'\in\Lambda_{r-1}$, by ($\ref{||<>}$),
($\ref{pdo}$), and ($\ref{sigma|D}$), we have
\begin{align*}
&\la\sigma^{\otimes[k/m]}\otimes s_J^{(r)}\otimes t_l^{(r)}\wedge
ds_D^{\otimes (k-1)}\ra^2_{ h_{S_{k-1}}\otimes h_r}\\
&\quad =\frac{ \la\sigma^{\otimes[k/m]}\otimes s_J^{(r)}\otimes
t_l^{(r)} \wedge ds_D^{\otimes (k-1)}\ra^2_{h\otimes h_r}}
{\sum\limits_{\substack{J'\in\Lambda_{r-1}\\l'=1,\dots,N}}\big|
\sigma^{\otimes[(k-1)/m]}\otimes s_{J'}^{(r-1)}\otimes
t_{l'}^{(r-1)}\wedge ds_D^{\otimes
(k-1)}\big|_h^2}\\
&\qquad = \frac{\big|\sigma^{\otimes[k/m]}\wedge ds_D^{\otimes
(k-1)}\big|_{h'}^2\big|s_{J_0'}^{(r-1)}\big|_{h^{(r-1)}}^2 \la
s_{r,j_0}\ra^2_{h_A\otimes
h_r}\big|t_l^{(r)}\big|_{h_A^{\otimes(m-r)}}^2}
{\sum\limits_{\substack{J'\in\Lambda_{r-1}\\l'=1,\dots,N}}
\big|\sigma^{\otimes[k/m]}\wedge ds_D^{\otimes (k-1)}\big|_{h'}^2
\big|s_{J'}^{(r-1)} \big|_{h^{(r-1)}}^2\big|t_{l'}^{(r-1)}
\big|_{h_A^{\otimes(m-r+1)}}^2} \\
&\qquad\quad =\frac{\big|s_{J_0'}^{(r-1)}\big|^2_{h^{(r-1)}}}
{\sum\limits_{J'\in\Lambda_{r-1}}\big|s_{J'}^{(r-1)}\big|_{h^{(r-1)}}^2}
\times \frac{\big|t_l^{(r)}
\big|_{h_A^{\otimes(m-r)}}^2}{\sum\limits_{l'=1}^N
\big|t_{l'}^{(r-1)}\big|_{h_A^{\otimes(m-r+1)}}^2}\la
s_{r,j_0}\ra^2_{h_A\otimes h_r}
\\
&\qquad \qquad
\leqslant\frac{\big|t_l^{(r)}\big|_{h_A^{\otimes(m-r)}}^2\la
s_{r,j_0}\ra^2_{h_A\otimes
h_r}}{\sum\limits_{l'=1}^N\big|t_{l'}^{(r-1)}
\big|_{h_A^{\otimes(m-r+1)}}^2}.
\end{align*}

By $(A_1)$ and the choices of $s_{r, j}$,
$$
C_1:=\max\limits_{l,r}\int_D\frac{\big|t_l^{(r)}
\big|_{h_A^{\otimes(m-r)}}^2\la s_{r,j}\ra_{h_A\otimes
h_r}^2}{\sum\limits_{l'=1}^N
\big|t_{l'}^{(r-1)}\big|_{h_A^{\otimes(m-r+1)}}^2}
$$
exists. It is clear that
$$
\int_D\la\sigma^{\otimes[k/m]}\otimes s_J^{(r)}\otimes t_l^{(r)}\wedge
ds_D^{\otimes (k-1)}\ra^2_{ h_{S_{k-1}}\otimes h_{r}}\leqslant C_1.
$$

{\bf Case 2}: $r=0$, i.e.~$[k/m]=[(k-1)/m]+1$. \\

We choose smooth metrics $h_A$ on $A|_D$, $h^{(m-1)}$ on
$(m-1)K_D+L^{(m-1)}|_D$, ${\hat h}$ on $K_D+L_m|_D+A|_D$, and $h'$
on $[(k-1)/m](mK_D+L^{(m)})+(k-1)(N_{D/X}^*+D|_D)$. We let
$h:=h'\otimes h^{(m-1)}\otimes h_A^{\otimes m}$ on $F_{k-1}|_D$.
Now $J \in \Lambda_0 = \{0\}$, by ($\ref{||<>}$), ($\ref{pdo}$),
and ($\ref{sigma|D}$), we have

\begin{align*}
&\la\sigma^{\otimes[k/m]}\otimes s_0^{(0)}\otimes t_l^{(0)}\wedge
ds_D^{\otimes (k-1)}\ra^2_{ h_{S_{k-1}}\otimes h_m}\\
&\quad=\frac{ \la\sigma^{\otimes[k/m]}\otimes t_l^{(0)}\wedge
ds_D^{\otimes (k-1)}\ra^2_{h\otimes h_m}}
{\sum\limits_{\substack{J'\in\Lambda_{m-1}\\l'=1,\dots,N}}
\big|\sigma^{\otimes[(k-1)/m]}\otimes s_{J'}^{(m-1)}\otimes
t_{l'}^{(m-1)}\wedge ds_D^{\otimes
(k-1)}\big|_h^2}\\
&\qquad = \frac{\big|\sigma^{\otimes[(k-1)/m]}\wedge ds_D^{\otimes
(k-1)}\big|_{h'}^2 \la \sigma\otimes
t_l^{(0)}\ra_{h^{(m-1)}\otimes h_A^{\otimes m}\otimes h_m}^2}
{\sum\limits_{\substack{J'\in\Lambda_{m-1}\\l'=1,\dots,N}}
\big|\sigma^{\otimes[(k-1)/m]}\wedge ds_D^{\otimes
(k-1)}\big|_{h'}^2 \big|s_{J'}^{(m-1)}\big|_{h^{(m-1)}\otimes
h_A^{\otimes(m-1)}}^2\big|t_{l'}^{(m-1)}\big|_{h_A}^2}\\
&\qquad \quad = \frac{\la \sigma\otimes
t_l^{(0)}\ra_{h^{(m-1)}\otimes h_A^{\otimes m}\otimes h_m}^2}
{\sum\limits_{J'\in\Lambda_{m-1}}\big|s_{J'}^{(m-1)}
\big|_{h^{(m-1)}\otimes h_A^{\otimes(m-1)}}^2
\sum\limits_{l'=1}^N\big|t_{l'}^{(m-1)}\big|_{h_A}^2}.
\end{align*}

By multiplying both the numerator and the denominator by the same
positive factor $\sum_{j = 1}^N |s_{m ,j}|^2_{\hat h}$, the
expression becomes

\begin{align*}
&\frac{\sum\limits_{j=1}^N\big|s_{m,j}\big|_{{\hat h}}^2
\la\sigma\otimes t_l^{(0)}\ra_{h^{(m-1)}\otimes h_A^{\otimes
m}\otimes h_m}^2}
{\sum\limits_{\substack{J'\in\Lambda_{m-1}\\
j=1,\dots,N}}\big|s_{J'}^{(m-1)}\big|_{h^{(m-1)}\otimes
h_A^{\otimes(m-1)}}^2\big|s_{m,j}\big|_{{\hat
h}}^2\sum\limits_{l'=1}^N\big|t_{l'}^{(m-1)}\big|_{h_A}^2}\\
&\quad= \frac{\sum\limits_{j=1}^N\la\sigma\otimes t_l^{(0)}\otimes
s_{m,j}\ra_{h^{(m-1)} \otimes {\hat h}\otimes h_A^{\otimes
m}\otimes h_m}^2} {\sum\limits_{J\in\Lambda_{m}^*}\big|\hat
s^{(m)}_J\big|_{h^{(m-1)} \otimes {\hat h}\otimes
h_A^{\otimes(m-1)}}^2\sum\limits_{l'=1}^N \big|t_{l'}^{(m-1)}
\big|_{h_A}^2}
\\
&\qquad = \frac{\big|\sigma\otimes
t_l^{(0)}\big|_{h^{(m-1)}\otimes {\hat h}\otimes h_A^{\otimes
(m-1)}}^2}{\sum\limits_{J\in\Lambda_{m}^*}\big|\hat
s^{(m)}_J\big|_{h^{(m-1)}\otimes {\hat h}\otimes
h_A^{\otimes(m-1)}}^2} \times \frac{\sum\limits_{j = 1}^N \la
s_{m,j}\ra_{h_A\otimes
h_m}^2}{\sum\limits_{l'=1}^N\big|t_{l'}^{(m-1)}\big|_{h_A}^2}.
\end{align*}

By $(A_1)$ and the choices of $s_{m,j}$,
$$C_2:=\sum\limits_{j=1}^N\int_D \frac{\la s_{m,j}\ra_{h_A\otimes
h_m}^2}{\sum\limits_{l'=1}^N\big|t_{l'}^{(m-1)}\big|_{h_A}^2}$$
exists. By $(A_2)$ and $(A_3)$, $\sigma\otimes t_l^{(0)}$ is a
linear combination of $\{\hat s^{(m)}_J\}_{J\in\Lambda_m^*}$. The
Cauchy--Schwartz inequality implies that
$$C_3:=\max\limits_{l=1,\dots,N}\sup\limits_D\frac{\big|\sigma\otimes
t_l^{(0)}\big|_{h^{(m-1)}\otimes {\hat h}\otimes h_A^{\otimes
(m-1)}}^2}{\sum\limits_{J\in\Lambda_{m}^*}\big|\hat
s^{(m)}_J\big|_{h^{(m-1)}\otimes {\hat h}\otimes
h_A^{\otimes(m-1)}}^2}
$$
exists. In this case we have
$$
\int_D\la\sigma^{\otimes[k/m]}\otimes s_0^{(0)}\otimes t_l^{(0)}\wedge
ds_D^{\otimes (k-1)}\ra^2_{ h_{S_{k-1}}\otimes h_m}\leqslant C_2 C_3.
$$
It is clear that $C_1$, $C_2$, $C_3$, and hence
$$
C':=\max\{C_1, C_2C_3\}
$$
depend only on $\sigma$ and the choices of $\{t_{l}^{(r)}\}$ and
$\{s_{j,l}\}$ in $(A_2)$ and $(A_3)$ (and is independent of
$k\geqslant m$ and the choices of the auxiliary metrics $h_A$,
$h^{(m-1)}$ and ${\hat h}$).

In summary,
$$
\int_D\la\sigma^{\otimes[k/m]}\otimes s_J^{(r)}\otimes
t_l^{(r)}\wedge ds_D^{\otimes (k-1)}\ra^2_{ h_{S_{k-1}}\otimes
h_{r^*}}\leqslant C'.
$$
By Theorem $\ref{otsa u}$, there exists a family of sections
$$
S_k=\{\widetilde{\sigma}_{J,l}^{(k)}:J \in\Lambda_k, 1 \leqslant l
\leqslant N\}
$$
of $F_k$ over $X$ such that
$$
\widetilde{\sigma}_{J,l}^{(k)}|_D = \sigma^{\otimes[k/m]} \otimes
s^{(r)}_{J} \otimes t_{l}^{(r)}\wedge ds_D^{\otimes k}
$$
and
$$
\large\int_X\sum_{\substack{J\in\Lambda_r\\
l=1,\dots,N}}\la\widetilde\sigma_{J,l}^{(k)}\ra^2_{h_D\otimes
h_{S_{k-1}}\otimes h_r}\leqslant C_0:=CC'.
$$
This completes the proof.
\end{proof}

\subsection{Siu's construction of the metric $h_\infty$}
For any $w_0=(w^1_0,\dots,w^n_0)\in\mathbf C^n$ and any $r>0$, we
let $D_r(w_0)$ denote $\{(w^1,\dots,w^n):|w^\nu-w^\nu_0|< r,
1\leqslant\nu\leqslant n\}$, the polydisk in $\mathbf C^n$
centered at $w_0$ with polyradii $(r,\dots,r)$. Choose a finite
open cover $\mathcal W'=\{W_\alpha'\}_{\alpha\in I}$ of $X$ such
that each $W_\alpha'$ is biholomorphic to $D_1(0)$ and $\mathcal
W=\{W_\alpha\}$ also covers $X$, where $W_\alpha\subseteq
W_\alpha'$ corresponds to $D_{1/3}(0)$. We also require that
$L_1|_{W_\alpha'},\dots,L_m|_{W_\alpha'}$, $D|_{W_\alpha'}$, and
$A|_{W_\alpha'}$ (and hence $F_k|_{W_\alpha'}$, $k\geqslant m$)
are trivial for all $\alpha\in I$.

Suppose that $\widetilde\sigma^{(k)}_{J,l}$ given by Lemma
$\ref{ind}$ are represented by holomorphic functions $\widetilde
f^{(k)}_{\alpha;J,l}$ on $W_\alpha'$ for each $\alpha\in I$.

\begin{lemma}\label{uni}
There exists $C_0'>0$ such that
$$
\max\limits_{\substack{x\in W_\alpha\\ \alpha\in
I}}\frac{1}{q}\log\sum_{l=1}^N\big|\widetilde
f^{(qm)}_{\alpha;0,l}(x)\big|^2\leq C_0'
$$
for all $q\in\mathbf N$.
\end{lemma}

The essential part of this result is the uniformity of $C_0'$ with
respect to $q\in\mathbf N$.

\begin{proof}
For each $x\in W_\alpha$ whose coordinate is
$w_x=(w_x^1,\dots,w_x^n)$, we let $W_x$ be the subset of
$W_\alpha'$ corresponding to $D_{1/3}(w_x)$. Since $\bigcup_{x\in
W_\alpha'}W_x\Subset W_\alpha'$, there exists $M>0$ such that on
all $W_x$ we have (following the notation in $\ref{aux})$
\begin{equation}\label{l2toun}
\frac{\sum\limits_{J,l} \big|f_{\alpha;,l}^{(k+1)}
\big|^2}{\sum\limits_{J',l'}\big| f_{\alpha;J',l'}^{(k)}
\big|^2}dV \leq M \sum\limits_{J,l} \la \widetilde
\sigma_{\alpha;J,l}^{(k+1)} \ra_{h_D\otimes h_{S_{k-1}}\otimes
h_{r^*}}^2
\end{equation}
where $u^\nu={\rm Re}\ w^\nu$, $v^\nu={\rm Im}\ w^\nu$, and
$dV=du^1\wedge dv^1\wedge\cdots\wedge du^n\wedge dv^n$. For each
$m\leqslant k\leqslant qm-1$, by Jensen's inequality,
($\ref{l2toun}$), and Lemma \ref{ind},
\begin{align*}
&\frac{1}{{\rm Vol} \big(D_{1/3}(w_x)\big)} \int_{D_{1/3}(w_x)}
\log\sum\limits_{J,l} \big|f_{\alpha;J,l}^{(k+1)} \big|^2dV
\\
&\qquad -\frac{1}{{\rm Vol} \big(D_{1/3}(w_x)\big)}
\int_{D_{1/3}(w_x)} \log\sum\limits_{J',l'}
\big|f_{\alpha;J',l'}^{(k)} \big|^2dV
\\
&\leqslant\log \Bigg(\frac{1}{{\rm
Vol}\big(D_{1/3}(w_x)\big)}\int_{D_{1/3}(w_x)}
\frac{\sum\limits_{J,l} \big|f_{\alpha;J,l}^{(k+1)}
\big|^2}{\sum\limits_{J',l'} \big|f_{\alpha;J',l'}^{(k)}
\big|^2}dV \Bigg)
\\
&\leqslant\log\bigg(\frac{1}{{\rm
Vol}(W_x)}\int_{W_x}\sum\limits_{J,l}\la
\widetilde\sigma_{\alpha;J,l}^{(k+1)}\ra_{h_D\otimes
h_{S_{k-1}}\otimes h_{r^*}}^2\bigg) \\
&\leqslant \log\frac{MC_0}{{\rm
Vol}(W_x)}.
\end{align*}

Summing up the above computation from $k=m$ to $k=qm-1$ and
applying the sub-mean value inequality, we obtain
\begin{align*}
&\frac{1}{q} \log\sum_{l=1}^N \big| \widetilde
f^{(qm)}_{\alpha;0,l}(x) \big|^2
\\
&\leqslant \frac{1}{q{\rm Vol} \big(D_{1/3}(w_x)\big)}
\int_{D_{1/3}(w_x)} \log\sum\limits_{l=1}^N \big|
f_{\alpha;0,l}^{(qm)} \big|^2 dV
\\
&\leqslant \frac{9^n}{q\pi^n} \int_{D_{1/3}(w_x)}
\log\sum\limits_{l=1}^N \big|f_{\alpha;0,l}^{(m)} \big|^2 dV +
\frac{(q-1)m}{q} \log\frac{9^nMC_0}{\pi^n}
\\
&\leqslant \frac{9^n}{q\pi^n} \int_{D_{2/3}(0)}
\log\sum\limits_{l=1}^N \big|f_{\alpha;0,l}^{(m)}\big|^2dV +
\frac{(q-1)m}{q}\log \frac{9^nMC_0}{\pi^n}.
\end{align*}

Since we have only finitely many $\alpha\in I$, the expected constant
$C_0'>0$ clearly exists.
\end{proof}

Now we are ready to construct the desired metric $h_\infty$.
\begin{lemma}\label{hinfty}
There exists a semipositive metric $h_\infty$ on
$m(K_X+D)+L^{(m)}$ such that $|\sigma \wedge ds_D^{\otimes m}|_{
h_\infty}\leqslant 1$.
\end{lemma}
\begin{proof}
On each $W_\alpha\in\mathcal W$ we let
$$
\widetilde f_\alpha^{(\infty)}:=\lim\limits_{p\to\infty}
\left(\sup\limits_{q \geqslant
p}\frac{1}{q}\log\sum_{l=1}^N\big|\widetilde
f^{(qm)}_{\alpha;0,l}\big|^2\right)^*
$$
where $(\ )^*$ denotes upper semicontinuous regularization. By
Lemma $\ref{uni}$,
$$
\left\{\left(\sup\limits_{q\geqslant p} \frac{1}{q}
\log\sum_{l=1}^N \big| \widetilde f^{(qm)}_{\alpha;0,l}
\big|^2\right)^*\right\}_{p\in\mathbf N}
$$
is a decreasing sequence of plurisubharmonic functions on $W_\alpha
'$ which are bounded above by $C_0'$ on $W_\alpha$, and hence
$\widetilde f_\alpha^{(\infty)}$ is also plurisubharmonic and
bounded from above by $C_0'$ on $W_\alpha$.

Let $g_{\alpha\beta}$ and $a_{\alpha\beta}\in\mathscr
O_X^*(W_\alpha\cap W_\beta)$, $\alpha,\beta\in I$ be the transition
functions of $m(K_X+D)+L^{(m)}$ and $mA$ respectively. By the
definition of $\{f^{(qm)}_{\alpha;0,l}\}$, we have
$$
f^{(qm)}_{\alpha;0,l}=\big(g_{\alpha\beta}\big)^qa_{\alpha\beta}
f^{(qm)}_{\beta;0,l}
$$
and hence
$$
\frac{1}{q}\log\sum_{l=1}^N\big|\widetilde
f^{(qm)}_{\alpha;0,l}\big|^2=\log|g_{\alpha\beta}|^2 + \frac{1}{q}\log|
a_{\alpha\beta}|^2+\frac{1}{q}\log\sum_{l=1}^N\big|\widetilde
f^{(qm)}_{\beta;0,l}\big|^2.
$$
Taking $\lim\limits_{p\to\infty}(\sup\limits_{q\geqslant
p}\underline{\ \ \ })^*$ to both sides and exponentiating them, we
get rid of the term involving $a_{\alpha\beta}$ and obtain
$$
e^{-\widetilde f^{(\infty)}_\beta}=|g_{\alpha\beta}|^2
e^{-\widetilde f^{(\infty)}_\alpha}.
$$
This shows that the set of local data $\{e^{-\widetilde
f^{(\infty)}_\alpha}:\alpha\in I\}$ defines a semipositive metric
$h_\infty$ on $m(K_X+D)+L^{(m)}$.

It remains to show that
$\big|\sigma\wedge ds_D^{\otimes m}\big|_{h_\infty}\leqslant 1$. By
Lemma $\ref{ind}$,
$$
\widetilde\sigma_{0,l}^{(qm)}|_D=\sigma^{\otimes q}\otimes
t_{l}^{(0)} \wedge ds_D^{\otimes(qm)}
$$
for $l=1,\dots,N$. Suppose that $\sigma\wedge ds_D^{\otimes m}$
and $t_{l}^{(0)}$ are represented by functions $\psi_\alpha$ and
$\tau_{\alpha;l}^{(0)}$ on $W_\alpha\cap D$ respectively. Then we
have $\widetilde
f^{(qm)}_{\alpha;0,l}=\psi_\alpha^q\tau^{(0)}_{\alpha;l}$ for each
$l$, and hence
$$
\frac{1}{q}\log\sum\limits_{l=1}^N\left.\big|
\widetilde f^{(qm)}_{\alpha;0,l}\big|^2\right|_{D\cap W_\alpha}\\
=\log|\psi_\alpha|^2+\frac{1}{q}\log\sum\limits_{l=1}^N\big|
\tau_{\alpha;l}^{(0)}\big|^2.
$$

Since
$$
\left.\left(\sup\limits_{q\geqslant p}\frac{1}{q}
\log\sum\limits_{l=1}^N\big|\widetilde f^{(qm)}_{\alpha;0,l}
\big|^2\right)^*\right|_{D\cap W_\alpha}\geqslant
\sup\limits_{q\geqslant p}\frac{1}{q}\log\sum\limits_{l=1}^N\left.\big|
\widetilde f^{(qm)}_{\alpha;0,l}\big|^2\right|_{D\cap W_\alpha},
$$
we obtain that
$$
\widetilde f^{(\infty)}_\alpha|_{W_\alpha\cap
D}=\lim\limits_{p\to\infty}\left. \left(\sup\limits_{q\geqslant
p}\frac{1}{q}\log\sum\limits_{l=1}^N\big| \widetilde
f^{(qm)}_{\alpha;0,l}\big|^2\right)^* \right|_{D\cap
W_\alpha}\geqslant\log|\psi_\alpha|^2.
$$
This shows that $e^{-\widetilde
f^{(\infty)}_\alpha}|\psi_\alpha|^2\leqslant 1$ for each $\alpha\in
I$ and hence completes the proof.
\end{proof}
{\ }\\

\section{Appendix 1} \label{app}
In this appendix we will provide a proof of Theorem $\ref{otsa
u}$. Let $\Omega$ be a complex manifold and let $D$ be a
nonsingular hypersurface in $\Omega$. Suppose that $(D, h_D)$ and
$(L, h)$ are line bundles on $\Omega$ with singular metrics, and
$s\in H^0(D, K_D+L|_D)$.

Consider the following statement:\\

\noindent $E\big(\Omega, (D, h_D), (L, h), s\big)$: If\\

(i) $(L,h)$ is semipositive,\\

(ii) $h|_D$ is well defined (see \ref{sm} and \ref{semip}),\\

(iii) there are real numbers $\mu > 0$ and $M>0$
such that
$$
\mu \sqrt{-1} \Theta_{h} \geqslant \sqrt{-1} \Theta_{h_D}
$$
as currents on $\Omega$ and
$${{\rm ess.}\sup}_{\Omega} |s_D|_{h_D} \leqslant M,$$
and

(iv)$$
\int_D \left<s\right>^2_{h|_D} < \infty ,
$$
then there is a section $\widetilde{s}_{\Omega}$ of $H^0(\Omega,
K_{\Omega} + D +L)$ such that $\widetilde{s}_{\Omega}|_D = s
\wedge ds_D$ and
$$
\int_{\Omega} \la\widetilde{s}_{\Omega}\ra_{h_D \otimes  h}^2
\leqslant C \int_D \la s \ra_{h}^2
$$
where $C>0$ only depends on $M$ and $\mu$.
\\

In order to simplify notations, when $L$ and $D$ are trivial line
bundles we always write $h=e^{-\kappa}$ and $h_D=e^{-\varphi_D}$,
and rewrite $E\big(\Omega, (D, h_D), (L, h), s\big)$ as $E(\Omega,
\varphi_D, \kappa, s)$. For brevity, when we write $E(\Omega,
\varphi_D, \kappa, s)$ we assume implicitly that $D$ and $L$ are
trivial bundles.

\begin{theorem} \label{otsa stein}
The statement $E(Y, \varphi_{D}, {\kappa}, s)$ holds if $Y$ is a
Stein manifold and $\varphi_D$ is the sum of a plurisubharmonic
function and a smooth function.
\end{theorem}

\begin{proof}[Proof of Theorem \ref{otsa u}]
Choose a sufficiently ample hypersurface $V$ in $X$ such that $D
\nsubseteq V$ and $D$ and $L$ are trivial over $X \backslash V$
and $\varphi_D$ is the sum of a plurisubharmonic function and a
smooth function (cf.\ Remark \ref{almost semip}). Then the theorem
follows from Theorem \ref{otsa stein} by taking $Y=X \backslash
V$, and the $L^2$ Riemann extension theorem.
\end{proof}

\subsection{Smoothing of singular metrics}

Let $Y$ be a Stein manifold of dimension $n$. Then we can find a
locally biholomorphic map $\pi : Y\rightarrow {\bf C}^n$
(cf.~\cite{gr}, p.225). In our case $Y$ will be $X \backslash V$,
the complement of an ample divisor in a projective manifold, for
which such a map $\pi$ can be constructed directly. The locally
biholomorphic map can be used to define the operation of
convolution for functions on relatively compact open subsets of
$Y$.

We define a function $R : Y \rightarrow {\bf R}^+ \cup
\{+\infty\}$ as follows. For $z\in{\bf C}^n$, denote by
$B_{R'}(z)$ the ball of radius $R'$ centered at $z$. For $y \in
Y$, let
\begin{center}
$R(y) := \sup\{\, R' > 0 \mid \pi: \pi^{-1} B_{R'}\big(\pi
(y)\big) \rightarrow B_{R'}\big(\pi (y)\big)$ is
biholomorphic$\,\}$.
\end{center}
$R$ is easily seen to be lower semicontinuous. Let
$$
R_{\Omega} := \inf_{y \in \Omega} R(y) > 0
$$
for every relatively compact open subset $\Omega \Subset Y$. Note
that $R_{\Omega} \geqslant R_{\Omega'}$ for $\Omega\Subset \Omega'
\Subset Y$. If $f : Y \rightarrow {\bf R} \cup \{- \infty\}$ is a
function and $\{ \rho_{\varepsilon} \}$ is a family of smoothing
kernels associated to a symmetric mollifier $\rho$ on ${\bf C}^n$,
then we can define the convolution $f_{\varepsilon}$ as follows.
For any $y \in Y$ we have a coordinate chart
$$
\pi_y:=\pi|_{U_y}: U_y:=\pi^{-1} B_{R(y)}\big(\pi (y)\big)
\rightarrow B_{R(y)}\big(\pi (y)\big).
$$
Then
$$
f_{\varepsilon}(y) := \big((f \circ \pi_y^{-1}) \ast
\rho_{\varepsilon}\big)\big(\pi (y)\big)
$$
for every $y$ with $R(y) > \varepsilon$. Note that for any $x,y\in
Y$, $f\circ\pi_x^{-1}|_{\pi(U_x\cap U_y)} =
f\circ\pi_y^{-1}|_{\pi(U_x\cap U_y)}$ and hence
$f_{\varepsilon}|_{U_y}=\big((f \circ \pi_y^{-1}) \ast
\rho_{\varepsilon}\big)\circ\pi|_{U_y}$. Therefore, if $f$ is
plurisubharmonic, the convolution $f_{\varepsilon}$ is also
plurisubharmonic on a relatively compact open subset $\Omega$ for
all $\varepsilon < R_{\Omega}$.

\begin{lemma}\label{peslem}
Suppose $D_0$ is a nonsingular
hypersurface in $Y$. If $D_0$ and $L$ are trivial bundles and
$\varphi_{D_0}$ and $\kappa_0$ are smooth on $Y$, then $E(\Omega,
\varphi_D, \kappa, s)$ holds for every relatively compact
pseudoconvex domain $\Omega$ with smooth boundary in $Y$ and
$$
s\in{\rm image}\big(H^0(D_0,K_{D_0}+L|_{D_0})\longrightarrow
H^0(D,K_D+L|_D)\big),
$$
where $D=D_0\cap\Omega$, $\varphi_D=\varphi_{D_0}|_\Omega$, and
$\kappa=\kappa_0|_\Omega$.
\end{lemma}

Now we deduce Theorem \ref{otsa stein} from Lemma \ref{peslem},
whose proof will be given in next subsection.

\begin{proof}[Proof of Theorem \ref{otsa stein}]
Suppose $\varphi_D=\varphi'+\varphi''$ where $\varphi'$ is
plurisubharmonic and $\varphi''$ is smooth, and suppose $s$ is a
section of $K_D + L|_D$ over $D$ with
$$
\int_D \left< s \right>^2_{h|_D} < \infty.
$$
Let $\varphi'_{\varepsilon} = \varphi'\ast \rho_{\varepsilon}$ and
$\varphi''_{\varepsilon} = \varphi''\ast \rho_{\varepsilon}$ on
the subdomain of $\Omega$ where they can be defined. We choose a
sequence of pseudoconvex domains $\Omega_1 \Subset \cdots \Subset
\Omega_{\nu} \Subset \Omega_{\nu+1} \Subset \cdots$ with smooth
boundary exhausting $Y$ and a decreasing sequence
$\{\varepsilon_{\nu}\}$ converging to zero such that the following
conditions hold:
\begin{enumerate}
\item $R_{\Omega_{\nu}} > \varepsilon_{\nu}$ and
$\kappa_{\varepsilon_{\nu}} = \kappa \ast
\rho_{\varepsilon_{\nu}}$ is a smooth plurisubharmonic function on
$\Omega_{\nu}$.

\item For each $N \in {\bf N}$, the sequences $\{
\kappa_{\varepsilon_{\nu}} \}_{\nu \geqslant N}$ and $\{
\varphi'_{\varepsilon_{\nu}} \}_{\nu \geqslant N}$ decrease to
$\kappa$ and $\varphi'$ on $\Omega_N$, respectively.

\item For each $\nu$, we have
$|s_D|^2e^{-\varphi''_{\varepsilon_\nu}}\leqslant 2
|s_D|^2e^{-\varphi''}$ on $\Omega_\nu$. (Here $|s_D|^2$ is taken
by viewing $s_D$ as a function via the global trivialization of
$D$. Note that on each relatively compact set
$e^{-\varphi''_{\varepsilon}}$ converges to $e^{-\varphi''}$ as
$\varepsilon\to 0$. Therefore for each $\nu$ we only need to
choose $\varepsilon_\nu$ so small that
$|e^{-\varphi''_{\varepsilon_\nu}}-e^{-\varphi''}| \leqslant
\inf\nolimits_{\Omega_\nu} e^{-\varphi''}$ on $\Omega_\nu$.)
\end{enumerate}

Therefore
$$
\sup\nolimits_{\Omega_{\nu}} |s_D|_{\varphi_{\varepsilon_{\nu}}}
\leqslant \sqrt{2}\,{{\rm ess.}\sup}_{\Omega} |s_D|_{\varphi_D}
\leqslant \sqrt{2} M
$$
for each $\nu$. Clearly, $\kappa_{\varepsilon_{\nu}}$ is not
identically $- \infty$ on $D \cap \Omega_{\nu}$.

The curvature condition
$$
\mu \sqrt{-1} \Theta_{\kappa} \geqslant \sqrt{-1} \Theta_{\varphi_D}
$$
implies that there is a plurisubharmonic function $\psi$ such that
$\mu \kappa - \varphi_D = \psi$ a.e.~on $Y$. Then $\mu
\kappa_{\varepsilon_{\nu}} - \varphi_{\varepsilon_{\nu}} = \psi
\ast \rho_{\varepsilon_{\nu}}$ a.e.~on $\Omega_{\nu}$. Since $\psi
\ast \rho_{\varepsilon_{\nu}}$ is plurisubharmonic, we get
$$
\mu \sqrt{-1} \Theta_{\kappa_{\varepsilon_{\nu}}} \geqslant
\sqrt{-1} \Theta_{\varphi_{\varepsilon_{\nu}}}
$$
on $\Omega_{\nu}$. Having assumed the validity of Lemma
\ref{peslem}, we can obtain such an extension
$\widetilde{s}_{\Omega_{\nu}}$. Since $\kappa_{\varepsilon_{\nu}}
\geqslant \kappa$, we obtain
\begin{equation}\label{ineq}
\int_{\Omega_N}
\la\widetilde{s}_{\Omega_{\nu}}\ra_{\varphi_{\varepsilon_{\nu}} +
\kappa_{\varepsilon_{\nu}}}^2 \leqslant C \int_{D \cap \Omega_N}
\la s\ra_{\kappa_{\varepsilon_{\nu}}}^2 \leqslant C \int_{D} \la
s\ra_{\kappa}^2
\end{equation}
for all $\nu \geqslant N$. (Here we abuse the notation by using
weight functions to stand for their associated metrics.) Notice
that the RHS is independent of $n$ (C only depends on $M$ and
$\mu$). By (iii) in $E(Y, \varphi_D, \kappa, s)$, for each
$N\in{\bf N}$, the weight function $\varphi + \kappa$ is bounded
from above on $\Omega_{N+1}$ by a number $M_N>0$. By the
definition of convolution $\varphi_{\varepsilon_{\nu}} +
\kappa_{\varepsilon_{\nu}}$ are bounded from above by the same
number $M_N$ on $\Omega_N$ for sufficiently large $\nu$. By
diagonal method we can select a subsequence
$\{\widetilde{s}_{\Omega_{\nu_k}}\}_{k \in {\bf N}}$ such that
$\{\widetilde{s}_{\Omega_{\nu_k}}\}_{k \geqslant N}$ converges
uniformly on $\Omega_{N}$ for each $N\in {\bf N}$. This way we
obtain a section $\widetilde{s}_{Y} \in H^0(Y, K_{Y} + D +L)$ by
setting $\widetilde{s}_{Y}|_{\Omega_{\nu_N}} = \lim_{k \rightarrow
\infty} \widetilde{s}_{\Omega_{\nu_k}}$. We let $\chi_{\Omega_N}$
be the characteristic function of $\Omega_N$ on $Y$. (\ref{ineq})
can be rephrased as
$$
\int_{Y} \chi_{\Omega_N}\la\widetilde{s}_{\Omega_{\nu}}
\ra_{\varphi_{\varepsilon_{\nu}} + \kappa_{\varepsilon_{\nu}}}^2
\leqslant C \int_{D} \la s\ra_{\kappa}^2.
$$
Applying Fatou's lemma, we obtain the desired inequality
$$
\int_{Y}\la\widetilde{s}_{Y}\ra_{\varphi + \kappa}^2 \leqslant C
\int_{D} \la s\ra_{\kappa}^2.
$$
\end{proof}

The rest of this appendix is devoted to proving Lemma
\ref{peslem}.

\subsection{Proof of Lemma \ref{peslem}}

From now on, we assume that $Y, \Omega$ and $s$ are as in Lemma \ref{peslem}. Let $\rho$ be a defining function of $\Omega$. We follow almost the same argument as Siu's in \cite{siu2}.

\begin{definition}
Let $(z^1,\dots,z^n)$ be local coordinates on some open set $U$
and let $e_{U}$ be a local holomorphic frame of $L$. We put $e^{-
\psi} = h(e_{U}, e_{U})$.
\begin{enumerate}
\item For $u, v$ being $L$-valued $(p,q)$-forms with measurable
coefficients, we set
$$
\left< u, v \right>_{h} := \left< u, v \right>_{g_{\omega} \otimes
h} dV_{\omega}
$$
and $\left|u\right|^2_{h}:=\left< u, u \right>_{h}$ where $g$ is a hermitian metric on
$\Omega$ with $\omega$ being its associated $(1, 1)$-form. We will
sometimes write $\left< u, v \right>_{\psi} = \left< u, v
\right>_{h}$ by abusing the notation. Note that when $(p,q)=(n,0)$
we have $|u|^2_{h}=\la u\ra^2_{h}$ as in Definition \ref{2/m}.

\item Given an $L$-valued $(n, 1)$-form $u$. Locally we have $u =
\sum_\beta u_{\overline{\beta}}\, e_{U} \otimes dz \wedge
d\overline{z}^{\beta}$ where $dz := dz^1 \wedge \cdots \wedge
dz^n$. We define an $(n, 0)$-form
$$
\iota^{\alpha}u := \sum_{\beta} g^{\alpha \overline{\beta}}
u_{\overline{\beta}}\, e_{U} \otimes dz.
$$
For a continuous $(1,1)$-form $\Xi$ which has a local expression
$\sqrt{-1} \xi _{\alpha\overline{\beta}} dz^{\alpha} \wedge
d\overline{z}^{\beta}$. We set $\Xi[u]_h := \sum_{\alpha,\beta} \xi
_{\alpha\overline{\beta}} \left< \iota^{\alpha}u, \iota^{\beta}u
\right>_{\psi}$.
\end{enumerate}
\end{definition}

We also need the following standard result from functional analysis:
\begin{lemma} \label{L2}
Let ${T} : \mathcal{H}_1 \rightarrow \mathcal{H}_2$ and ${S} :
\mathcal{H}_2 \rightarrow \mathcal{H}_3$ be closed, densely
defined operators between Hilbert spaces with $ST=0$, and let
$C>0$ be a constant. Given $g \in \mathcal{H}_2$ with $Sg = 0$.
Then there exists $v\in \mathcal{H}_1$ such that ${T} v = g$ and
$\| v \| \leqslant C$ if and only if
\begin{equation} \label{funal}
|(u,g)|^2 \leqslant C^2(\|{T}^*u\|^2 + \|{S}u\|^2)
\end{equation}
for all $u\in \mathrm{Dom}\,{S} \cap \mathrm{Dom}\,{T}^*$.
\end{lemma}

Let $\varphi$, $\eta$ and $\gamma$ be smooth functions with $\eta,
\gamma > 0$. Set $\eta e^{-\psi} = e^{-\varphi}$. We recall the
twisted Bochner--Kodaira formula (see \cite{siu2}, Proposition
3.4)
\begin{equation} \label{tw bk}
\begin{split}
\int_{\Omega} |\overline{\partial}^*_{\psi}{u}|^2_{\varphi} &+
\int_{\Omega} |\overline{\partial}{u}|^2_{\varphi} =
\int_{\partial{\Omega}} \sqrt{-1} \partial
\overline{\partial}\rho_\Omega[u]_{\varphi} + \int_{\Omega}
|\nabla^{0,1} u|^2_{\varphi} \\
&\quad + \int_{\Omega} \left(\eta \sqrt{-1} \partial
\overline{\partial} \psi - \sqrt{-1} \partial \overline{\partial}
\eta\right)[u]_{\psi} + 2 \mathrm{Re}{\int_{\Omega}
\left<\iota^{\partial{\eta}}u,
\overline{\partial}^*_{\psi}{u}\right>_{\psi}}
\end{split}
\end{equation}
for each $L$-valued smooth $(n, 1)$-form $u$ in
$\mathrm{Dom}\,\overline{\partial}^*_{\psi} \cap
\mathrm{Dom}\,\overline{\partial}$.

\begin{remark}\label{dens gr}
For $\mathcal{E}_c^{n, 1}\left(\overline{\Omega}, L \right)$
being the space of $L$-valued smooth $(n, 1)$-forms with compact
supports, $\mathcal{E}_c^{n, 1} \left(\overline{\Omega}, L\right) \cap \mathrm{Dom}\,\overline{\partial}^*_{\psi} \cap
\mathrm{Dom}\,\overline{\partial}$ is dense in
$\mathrm{Dom}\,\overline{\partial}^*_{\psi} \cap
\mathrm{Dom}\,\overline{\partial}$ with respect to the graph norm.
Therefore, to get a priori estimate from Lemma $\ref{L2}$ we only
need to consider smooth $u$ with compact supports.
\end{remark}

Since $\Omega$ is pseudoconvex, the Levi form of $\rho_\Omega$ is
semipositive at each point of $\partial{\Omega}$. Adding
$\int_{\Omega} \gamma |\overline{\partial}^*_{\psi}{u}|^2_{\psi}$
to both side of (\ref{tw bk}) and using $\eta e^{-\psi} =
e^{-\varphi}$, the twisted Bochner--Kodaira formula becomes
\begin{equation} \label{tw bk2}
\begin{split}
&\int_{\Omega} \left(\eta + \gamma\right)
|\overline{\partial}^*_{\psi}{u}|^2_{\psi} + \int_{\Omega} \eta
|\overline{\partial}{u}|^2_{\psi} \\
&\geqslant \int_{\Omega}
\left(\eta \sqrt{-1} \partial \overline{\partial} \psi - \sqrt{-1}
\partial \overline{\partial} \eta\right)[u]_{\psi} + 2
\mathrm{Re}{\int_{\Omega} \left<\iota^{\partial{\eta}}u,
\overline{\partial}^*_{\psi}{u}\right>_{\psi}} + \int_{\Omega}
\gamma |\overline{\partial}^*_{\psi}{u}|^2_{\psi}.
\end{split}
\end{equation}

We set $r(x) := |s_D (x)|_{h_D}$ for $x \in \Omega$. We first assume that $\tfrac{1}{\mu}
\geqslant 2 M^2$ and let $c$ be a positive constant to be
specified later. We set $N_0 := \max\{1, \sqrt{e} M^{2 c}\}$.
Choose any positive number $A > N_0$. Let
$$
\varepsilon_0 = \sqrt{\left(\tfrac{A}{\sqrt{e}}\right)^{1/c} - M^2}.
$$
For each positive $\varepsilon \leqslant \varepsilon_0$, we let
$$
\eta = \log \frac{A}{(r^2 + \varepsilon^2)^c}
$$
and
$$
\gamma = \frac{2 c^2}{r^2 + \varepsilon^2}.
$$
Then $\eta \geqslant 1/2$ on $\Omega$. Applying the Cauchy--Schwarz
inequality and $\partial{\eta} = - \tfrac{2 c r}{r^2 + \varepsilon^2}
\partial{r}$, we obtain
\begin{align*}
\left|2{\mathrm{Re}}\int_{\Omega} \left< \iota^{\partial{\eta}}u,
\overline{\partial}^*_{\psi}{u}\right>_{\psi} \right| &\leqslant 2
\int_{\Omega} |\iota^{\partial{\eta}}u|_{\psi}
|\overline{\partial}^*_{\psi}{u}|_{\psi}
\\
&=2\int_{\Omega} \frac{2 c r}{r^2 + \varepsilon^2}
|\iota^{\partial{r}}u|_{\psi}
|\overline{\partial}^*_{\psi}{u}|_{\psi}\\
&=\int_{\Omega} \frac{2 r^2}{r^2 + \varepsilon^2}
|\iota^{\partial{r}}u|_{\psi}^2 + \int_{\Omega} \frac{2
c^2}{r^2 + \varepsilon^2}| \overline{\partial}^*_{\psi}{u}|_{\psi}^2\\
&=\int_{\Omega} \frac{2 r^2}{r^2 + \varepsilon^2}
|\iota^{\partial{r}}u|_{\psi}^2 + \int_{\Omega}\gamma
|\overline{\partial}^*_{\psi}{u}|_{\psi}^2.
\end{align*}
From (\ref{tw bk2}) it follows that
\begin{equation} \label{t oti}
\begin{split}
\int_{\Omega} \left(\eta + \gamma\right)
&|\overline{\partial}^*_{\psi}{u}|^2_{\psi} + \int_{\Omega} \eta
|\overline{\partial}{u}|^2_{\psi}
\\
&\geqslant \int_{\Omega}
\left(\eta \sqrt{-1} \partial \overline{\partial} \psi - \sqrt{-1}
\partial \overline{\partial} \eta\right)[u]_{\psi} - \int_{\Omega}
\frac{2 r^2}{r^2 + \varepsilon^2} |\iota^{\partial{r}}u|_{\psi}^2
\end{split}
\end{equation}
Now we compute $-\partial\overline{\partial}{\eta}$. Since $r^2
\partial\overline{\partial}{\log{r^2}} = 2 r
\partial\overline{\partial}{r} - 2 \partial{r} \wedge
\overline{\partial}{r}$, it follows that
\begin{align*}
\sqrt{-1}\partial\overline{\partial}{r^2}&= 2 \sqrt{-1}
\partial{r} \wedge \overline{\partial}{r} + 2 r
\sqrt{-1}\partial\overline{\partial}{r}
\\
&= r^2\sqrt{-1}
\partial\overline{\partial}{\log{r^2}} + 4 \sqrt{-1} \partial{r}
\wedge \overline{\partial}{r}.
\end{align*}
By the Poincar\'e--Lelong formula,
$$
\sqrt{-1} \partial\overline{\partial}{\log{r^2}} = 2 \pi [D] -
\sqrt{-1} \partial\overline{\partial}{\varphi_D},
$$
where $[D]$ is the current of integration over $D$. Hence
\begin{equation} \label{r2 1}
\sqrt{-1} \partial\overline{\partial}{r^2} = 4 \sqrt{-1} \partial{r}
\wedge \overline{\partial}{r} - r^2 \sqrt{-1}
\partial\overline{\partial}{\varphi_D}.
\end{equation}
The term involving the current of integration vanishes since $r^2
\equiv 0$ on $D$.

We let $\eta_0 = - \log{(r^2 + \varepsilon^2)}$. From
$\partial\overline{\partial}{r^2} =
\partial\overline{\partial}{\left(e^{-\eta_0}\right)}$ it follows that
\begin{equation}\label{r2 2}
\begin{split}
\partial\overline{\partial}{r^2} &= e^{-\eta_0} \left(\partial{\eta_0}
\wedge \overline{\partial}{\eta_0} -
\partial\overline{\partial}{\eta_0}\right) \\ &= \frac{4r^2}{r^2 +
\varepsilon^2} \partial{r} \wedge \overline{\partial}{r} -
\big(r^2 + \varepsilon^2\big) \partial\overline{\partial}{\eta_0}.
\end{split}
\end{equation}
Using (\ref{r2 1}), (\ref{r2 2}) and
$\partial\overline{\partial}{\eta} = c
\partial\overline{\partial}{\eta_0}$, we get
\begin{equation} \label{eta}
- \sqrt{-1} \partial\overline{\partial}{\eta} = - \frac{c r^2}{r^2
+ \varepsilon^2} \sqrt{-1} \partial\overline{\partial}{\varphi_D}
+ \frac{4c \varepsilon^2}{\left(r^2 + \varepsilon^2\right)^2}
\sqrt{-1} \partial{r} \wedge \overline{\partial}{r}.
\end{equation}

Choose $\psi = \kappa + \tfrac{r^2}{2 \mu M^2}$. Note that $\sqrt{-1}
\partial \overline{\partial} \psi \geqslant 0$. Using ($\ref{r2 1}$) and
($\ref{eta}$) we get
\begin{equation}\label{phi}
\begin{split}
\eta &\sqrt{-1} \partial\overline{\partial}{\psi} - \sqrt{-1}
\partial\overline{\partial}{\eta} =  \eta \sqrt{-1}
\partial\overline{\partial}{\kappa}
\\
& -\eta \left(\frac{r^2}{2 \mu M^2} + \frac{c r^2}{\eta \left(r ^2
+ \varepsilon^2\right)} \right) \sqrt{-1}
\partial\overline{\partial}{\varphi_D}+ \left(\frac{4 \eta}{2 \mu
M^2} + \frac{4c \varepsilon^2}{\left(r^2 +
\varepsilon^2\right)^2}\right) \sqrt{-1} \partial{r} \wedge
\overline{\partial}{r}.
\end{split}
\end{equation}
Since $\eta \geqslant 1/2$ we get
\begin{equation}\label{cinq}
\frac{r^2}{2 \mu M^2} + \frac{c r^2}{\eta \left(r ^2 +
\varepsilon^2\right)} \leqslant \frac{1}{2 \mu} + 2 c.
\end{equation}
We now choose $c$ so that $c \leqslant \tfrac{1}{4 \mu}$. It
follows that
\begin{equation}\label{kappa}
\sqrt{-1} \partial\overline{\partial}{\kappa} - \left(\frac{r^2}{2
\mu M^2} + \frac{c r^2}{\eta \left(r ^2 + \varepsilon^2\right)}
\right) \sqrt{-1} \partial\overline{\partial}{\varphi_D} \geqslant
0
\end{equation}
where the inequality is from (\ref{cinq}) and the curvature hypothesis.

From ($\ref{t oti}$), ($\ref{phi}$), ($\ref{kappa}$) and $\tfrac{4
\eta}{2 \mu M^2} \geqslant 2 \geqslant \tfrac{2 r^2}{r^2 +
\varepsilon^2}$ we obtain
\begin{equation} \label{pri inq}
\int_{\Omega} \left(\eta + \gamma\right)
|\overline{\partial}^*_{\psi}{u}|^2_{\psi} + \int_{\Omega} \eta
|\overline{\partial}{u}|^2_{\psi} \geqslant \int_{\Omega} \frac{4
c \varepsilon^2}{\left(r^2 + \varepsilon^2\right)^2}
|\iota^{\partial{r}}u|_{\psi}^2.
\end{equation}

We now consider the modified $\overline{\partial}$ operators ${T}$
and ${S}$ defined by
\begin{center}
${T} u = \overline{\partial}\left({\sqrt{\eta + \gamma} u}\right)
\qquad $ and $\qquad {S} u = \sqrt{\eta} (\overline{\partial}u )$,
\end{center}
respectively. They are densely defined and ${S} \circ {T} = 0$,
and we can rewrite (\ref{pri inq}) to obtain the following lemma.
(See Remark $\ref{dens gr}$.)

\begin{lemma} \label{funal inq}
For each $L$-valued $(n, 1)$-form $u$ in $\mathrm{Dom}\,{S} \cap
\mathrm{Dom}\,{T}^*$ we have
\begin{equation} \label{TSeq}
\| {T}^*u\|_{\Omega, \psi}^2 + \| {S} u\|_{\Omega, \psi}^2
\geqslant \int_{\Omega} \frac{4 c \varepsilon^2}{\left(r^2 +
\varepsilon^2\right)^2} |\iota^{\partial{r}}u|_{\psi}^2.
\end{equation}
Here $\|\cdot\|_{\Omega, \psi}$ means the $L^2$ norm for $\Omega$ with
respect to the weight function $e^{-\psi}$.
\end{lemma}

Since $Y$ is Stein, there exists a $(D + L)$-valued $n$-form
$\widetilde{s}_0$ on $Y$ such that $\widetilde{s}_0|_{D} = s
\wedge ds_D$. Choose any number $0 < \delta < 1$. Let $\varrho \in
C^{\infty}([0, + \infty))$ be a cut-off function with $0 \leqslant
\varrho(x) \leqslant 1$ so that $\varrho$ is identically 1 on $[0,
\tfrac{\delta}{2}]$ and
\begin{equation}\label{cut}
\mathrm{supp}\,\varrho \subseteq[0,1] \quad \text{and}
\quad\sup|\varrho^{\prime}| \leqslant 1 + \delta.
\end{equation}
Let $\varrho_{\varepsilon} :=
\varrho\left(\tfrac{r^2}{\varepsilon^2}\right)$ and let
$$
\alpha_{\varepsilon} := \frac{2 r}{\varepsilon^2}
\,\varrho^{\prime} \left(\frac{r^2}{\varepsilon^2}\right)
\overline{\partial}{r} \wedge \left(s_D^{-1} \otimes
\widetilde{s}_0\right).
$$
Note that $\alpha_{\varepsilon}$ is smooth because the singularity
of $\overline{\partial}{r} \wedge \left(s_D^{-1} \otimes
\widetilde{s}_0\right)$ lies in the zero locus $D$ of $s_D$ and
$\varrho^{\prime} \left(\tfrac{r^2}{\varepsilon^2}\right)$ equals
zero in the tubular neighborhood $r^2 < \tfrac{\delta}{2}
\varepsilon^2$. Then we have
\begin{align*}
|(u, \alpha_{\varepsilon})_{\Omega,\psi}|^2 &= \left(\int_{\Omega}
\left| \left< u, \frac{2r}{\varepsilon^2} \,\varrho^{\prime}
\left(\frac{r^2}{\varepsilon^2}\right) \overline{\partial}{r}
\wedge \left(s_D^{-1} \otimes
\widetilde{s}_0\right)\right>_{\psi} \right| \right)^2 \\
&=\left(\int_{\Omega} \left|\left<\iota^{\partial{r}}u,
\frac{2r}{\varepsilon^2}
\,\varrho^{\prime}\left(\frac{r^2}{\varepsilon^2}\right) s_D^{-1}
\otimes \widetilde{s}_0\right>_{\psi}\right| \right)^2 \\
&\leqslant\left(\int_{\Omega} 2 |\iota^{\partial{r}}u|_{\psi}
\left|\frac{r}{\varepsilon^2}
\,\varrho^{\prime}\left(\frac{r^2}{\varepsilon^2}\right) s_D^{-1}
\otimes
\widetilde{s}_0\right|_{\psi} \right)^2 \\
&\leqslant\left(\int_{\Omega} \left|\frac{r}{\varepsilon^2}
\,\varrho^{\prime}\left(\frac{r^2}{\varepsilon^2}\right) s_D^{-1}
\otimes \widetilde{s}_0\right|_{\psi}^2 \frac{(r^2 +
\varepsilon^2)^2}{c \varepsilon^2} \right) \\
&\quad \times
\left(\int_{\Omega} \frac{4 c \varepsilon^2}{\left(r^2 +
\varepsilon^2\right)^2} |\iota^{\partial{r}}u|_{\psi}^2\right)\\
&\leqslant C_{\varepsilon} \left(\|T^*u\|_{\Omega, \psi}^2 +
\|Su\|_{\Omega,\psi}^2\right),
\end{align*}
where the last inequality is from Lemma \ref{funal inq}, and we
have used the notation
$$
C_{\varepsilon} := \int_{\Omega} \left|\frac{r}{\varepsilon^2}
\,\varrho^{\prime}\left(\frac{r^2}{\varepsilon^2}\right) s_D^{-1}
\otimes \widetilde{s}_0\right|_{\psi}^2 \frac{(r^2 +
\varepsilon^2)^2}{c \varepsilon^2}.
$$
By Lemma ($\ref{L2}$), we can solve the equation ${T}
\beta_{\varepsilon} = \bar{\partial}\left(\sqrt{\eta + \gamma}
\beta_{\varepsilon}\right)= \alpha_{\varepsilon}$ such that
\begin{equation} \label{heps}
\int_{\Omega} |\beta_{\varepsilon}|^2_{\psi} \leqslant
C_{\varepsilon}.
\end{equation}
\subsection{Estimate the constant $C_{\varepsilon}$}
Now we estimate the constant $C_{\varepsilon}$. Take $y \in Y$ an
arbitrary point and $(z^j = x^j +i y^j)$ local coordinates on a
open set $U_y$ centered at $y$, and let $e_L$ (respectively,
$e_D$) be local frames of $L$ (respectively, $D$) such that the
following conditions holds:
\begin{enumerate}
\item $s_D = z^n \otimes e_D$ on $U_y$;

\item If $\zeta = \xi + i \tau := z^n e^{-\frac{\varphi_D}{2}}$,
then $(x^1, y^1, \cdots, \xi, \tau)$ forms a coordinate system;

\item $U_y = P_{n-1} \times \{r < \varepsilon\}$ where $P_{n-1}$
is a $(n-1)$-dimensional polydisc;

\item We have
\begin{center}
$\widetilde{s}_0 = \widetilde{\sigma}_{U_y} e_D \otimes e_L
\otimes dz^1 \wedge \cdots \wedge dz^n$ and $s = \sigma_{U_y} e_L
\otimes dz^1 \wedge \cdots \wedge dz^{n-1}$
\end{center}
on $U_y$. (Note that $\widetilde s_0$ is defined not only on
$\Omega$ but on $Y$.)
\end{enumerate}
Since $\widetilde{s}_0|_D = s \wedge ds_D$, we get
$$
\widetilde{\sigma}_{U_y}(z^1, \cdots, z^{n-1}, 0) = \sigma_{U_y}(z^1,
\cdots, z^{n-1})
$$
on $U_y$. Choose a partition of unity $\{\rho_j\}$ subordinate to
a finite subcover $\{U_j\}\subset\{U_y\}_{y\in\overline\Omega}$ of
$\overline\Omega$. Then
\begin{align*}
C_{\varepsilon}&=\int_{\Omega} \frac{(r^2 + \varepsilon^2)^2}{c
\varepsilon^6} r^2
\Big|\varrho^{\prime}\left(\frac{r^2}{\varepsilon^2}\right)\Big|^2
| s_D^{-1} \otimes \widetilde{s}_0|_{\psi}^2 \\
&\leqslant\frac{(1 + \delta)^2}{c} \int_{\Omega \cap
\{\sqrt{\frac{\delta}{2}} \varepsilon \leqslant r \leqslant
\varepsilon\}} \frac{(r^2 +
\varepsilon^2)^2}{\varepsilon^6} r^2 | s_D^{-1} \otimes
\widetilde{s}_0|_{\psi}^2.
\end{align*}
For each $j$ we let $V_j := U_j \cap \Omega \cap
\{\sqrt{\tfrac{\delta}{2}} \varepsilon \leqslant r \leqslant
\varepsilon\}$ and
\begin{align*}
I_j&=\int_{V_j} \rho_j \frac{(r^2 +
\varepsilon^2)^2}{\varepsilon^6} r^2 | s_D^{-1} \otimes
\widetilde{s}_0|_{\psi}^2\\
&=\int_{V_j} \rho_j\frac{(r^2 + \varepsilon^2)^2}{\varepsilon^6}
\big|\widetilde{\sigma}_{U_j}\big|^2 e^{- \kappa - \tfrac{r^2}{2
\mu M^2}} e^{- \varphi_D} dx^1 \wedge dy^1 \wedge \cdots \wedge
dy^n.
\end{align*}
Therefore
$$
C_{\varepsilon} \leqslant \frac{(1 + \delta)^2}{c}
\sum\nolimits_j I_j.
$$
A direct computation yields
\begin{align*}
&e^{- \varphi_D} dx^1 \wedge dy^1 \wedge \cdots \wedge dy^n \\
&=
\left(1 + O(r)\right) dx^1 \wedge dy^1 \wedge \cdots \wedge dy^{n -
1} \wedge d\xi \wedge d\tau.
\end{align*}

Let $f := \rho_j \big|\widetilde{\sigma}_{U_j}\big|^2
e^{-\kappa}$. Then
\begin{align*}
I_j&\leqslant\int_{V_j} \rho_j
\big|\widetilde{\sigma}_{U_j}\big|^2 e^{-\kappa} \frac{(r^2 +
\varepsilon^2)^2}{\varepsilon^6} (1 + O(r)) r
\,dx^1 dy^1 \cdots dy^{n - 1} dr\, d\theta \\
&= \mathrm{I}^{(j)} +
\mathrm{II}^{(j)} + \mathrm{III}^{(j)},
\end{align*}
where
$$
\mathrm{I}^{(j)} = \int_{P_{n-1} \cap D} f(z^1, \cdots, z^{n-1},
0) \,dx^1 dy^1 \cdots dy^{n-1} \bigg(2 \pi \int_0^{\varepsilon}
\frac{(r^2 + \varepsilon^2)^2}{\varepsilon^6} r \, dr \bigg),
$$
$$
\mathrm{II}^{(j)} = \int_{V_j} (f(z^1, \cdots, z^n) - f(z^1,
\cdots, 0)) \frac{(r^2 + \varepsilon^2)^2}{\varepsilon^6} r \,dx^1
dy^1 \cdots dr\, d\theta,
$$
$$
\mathrm{III}^{(j)} = \int_{V_j} f(z^1, \cdots, z^n) \frac{(r^2 +
\varepsilon^2)^2}{\varepsilon^6} O(r) r \,dx^1 dy^1 \cdots dr\,
d\theta.
$$
Note that the term $f(z^1, \cdots, z^n) - f(z^1, \cdots, 0)$ in
$\mathrm{II}^{(j)}$ produces a factor $r$. Thus
$\mathrm{II}^{(j)}$ and $\mathrm{III}^{(j)}$ converge to zero as
$\varepsilon$ tend $0^+$. Then
\begin{equation} \label{ceps}
\begin{split}
  \limsup_{\varepsilon \rightarrow 0^+} C_{\varepsilon} & \leqslant
\frac{(1 + \delta)^2}{c} \sum\nolimits_j \mathrm{I}^{(j)} \\
&= \frac{2 \pi (1 + \delta)^2}{c} \Big(\int_{\Omega \cap D} \la
s\ra_h^2 \Big) \limsup_{\varepsilon \rightarrow 0^+}
\int_0^{\varepsilon} \frac{(r^2 + \varepsilon^2)^2}{\varepsilon^6}
r \, dr\
\\
&= \frac{7 \pi}{3 c} (1 + \delta)^2 \int_{\Omega \cap D} \la
s\ra_h^2.
\end{split}
\end{equation}

\subsection{Extension and the $L^2$ norm bound}
Now we set
$$
\widetilde{S}_{\varepsilon} := \varrho_{\varepsilon} \widetilde{s}_0 -
\sqrt{\eta + \gamma} \left(s_D \otimes \beta_{\varepsilon}\right).
$$
Then $\widetilde{S}_{\varepsilon}$ is a holomorphic section by
construction and
\begin{center}
$\int_{\Omega}  |\varrho_{\varepsilon} \widetilde{s}_0|_{h_D
\otimes h}^2 \rightarrow 0$ as $\varepsilon \rightarrow 0^+$,
\end{center}
because $\widetilde{s}_0$ is smooth in the relatively compact set
$\Omega$ and the support of $\varrho_{\varepsilon}
\widetilde{s}_0$ approaches a set of measure zero in $\Omega$ as
$\varepsilon \rightarrow 0^+$.

The supremum norm of $r \sqrt{\eta + \gamma}$ on $\Omega \subseteq
\{r \leqslant M\}$ is no more than the square root of
\begin{equation*}
\sup_{0 < x \leqslant M} x^2 \left(\log A + c \log\frac{1}{x^2 +
\varepsilon^2} + \frac{2 c^2}{x^2 + \varepsilon^2}\right)
\leqslant M^2 \log A + \frac{c}{e} + 2 c^2,
\end{equation*}
because the maximum of $y \log\tfrac{1}{y}$ on $(0, + \infty)$
occurs at $y = \tfrac{1}{e}$.

Take $A \rightarrow N_0^+$, $\delta \rightarrow 0^+$. By using
(\ref{ceps}) and
\begin{equation*}
\int_{\Omega} |\beta_{\varepsilon}|^2_{h} \leqslant e^{\frac{1}{2
\mu}} C_{\varepsilon}
\end{equation*}
from (\ref{heps}) and $r \leqslant M$, we get
$$
\limsup_{\varepsilon\rightarrow 0^+} \int_{\Omega}
\la\widetilde{S}_{\varepsilon}\ra_{h_D \otimes h}^2 \leqslant C_0
\int_{\Omega} \la s\ra_h^2
$$
where $C_0 = \tfrac{7 \pi}{3} e^{\frac{1}{2 \mu}}
\sqrt{\left(\tfrac{M}{c}\right)^2 \log N_0 + \tfrac{1}{ce} + 2}$.
Then the limit $\widetilde{s}_{\Omega}$ (up to subsequences) is an
$D + L$-valued holomorphic $n$-form on $\Omega$ whose restriction
to $D$ is $s \wedge ds_D$ with the following estimate
$$
\int_\Omega \la\widetilde{s}_{\Omega}\ra_{h_D \otimes h}^2
\leqslant C_0 \int_D \la s\ra_h^2.
$$

If $\tfrac{1}{\mu} < 2 M^2$, we replace the metric $h_D$ by the
metric $h_D^{\prime} := \tfrac{1}{2 \mu M^2} h_D$. Then
$\sup\nolimits_\Omega |T|_{h_D^{\prime}} = \tfrac{1}{\sqrt{2
\mu}}$. This finishes the proof of Theorem \ref{otsa stein}.

\begin{remark}
In the statement of Theorem \ref{otsa stein} the requirement that
$D$ and $L$ being trivial bundles is used only for smoothing the
metrics on them. Therefore the same argument shows that $E\big(\Omega,
(D,h_D), (L,h), s\big)$ holds if $\Omega$ is Stein, $(D,h_D)$ and
$(L,h)$ are smoothly metrized, and $s\in H^0(D,K_D+L|_D)$.
\end{remark}

\section{Appendix 2}
The following lemma about generalized multiplication maps is used
in \ref{amp} to select the auxiliary ample divisor to fulfill
$(A_3)$. For the convenience of the readers we give a proof in
this appendix. Some of its special cases are well known in
\cite{laza}, \cite{vieh}. The proof presented below is a
modification of their arguments.

\begin{lemma} \label{s.lem}
Let $D$ and $E$ be ample Cartier divisors on a scheme $X$. For any
coherent sheaves $\mathscr{F}_1$ and $\mathscr{F}_2$ on $X$, there
is a positive integer $m_0=m_0(D,E,\mathscr{F}_1,\mathscr{F}_2)$
such that
\begin{multline*}
H^0\big(X, \mathscr{F}_1 \otimes \mathscr{O}_X(aD)\big)\otimes
H^0\big(X, \mathscr{F}_2 \otimes \mathscr{O}_X(bE)\big)\rightarrow \\
H^0\big(X, \mathscr{F}_1 \otimes \mathscr{F}_2\otimes
\mathscr{O}_X(aD+bE)\big)
\end{multline*}
is surjective for all $a,b\geqslant m_0$.
\end{lemma}

\begin{proof}
First we assume that $\mathscr{F}_1$ and $\mathscr{F}_2$ are
locally free. Consider on $X\times X$ the exact sequence
\begin{equation}\label{diag}
0\longrightarrow \mathscr{I}_{\Delta} \longrightarrow
\mathscr{O}_{X\times X} \longrightarrow \Delta_{\ast}\mathscr{O}_X
\longrightarrow 0
\end{equation}
where $\Delta$ is the diagonal morphism. Let $p_1$ and $p_2$ be
the two projections and
$$aD\boxplus bE=p_1^{\ast}(aD)\otimes p_2^{\ast}(bE)$$
and
$$
\mathscr{G}=p_1^{\ast}\mathscr{F}_1\otimes p_2^{\ast}\mathscr{F}_2.
$$
By tensoring $(\ref{diag})$ with $\mathscr{G}$, we get
\begin{equation*}
0\longrightarrow \mathscr{G}\otimes\mathscr{I}_{\Delta}
\longrightarrow \mathscr{G} \longrightarrow
\mathscr{G}\otimes\Delta_{\ast}\mathscr{O}_X \longrightarrow 0.
\end{equation*}
Twisting by $\mathscr{O}_{X\times X}(aD\boxplus bE)$ and taking
cohomology, we obtain an exact sequence
\begin{multline*}
H^0\big(X\times X, \mathscr{G}(a,b)\big) \rightarrow
H^0\big(X\times X,
(\mathscr{G}\otimes\Delta_{\ast}\mathscr{O}_X)(a,b)\big)
\rightarrow\\
H^1\big(X\times X,
(\mathscr{G}\otimes\mathscr{I}_{\Delta})(a,b)\big)
\end{multline*}
where we use $(a,b)$ to denote the twisting
$\otimes\mathscr{O}_{X\times X}(aD\boxplus bE)$. It suffices to
verify that there is a positive integer $m_0$ such that
\begin{equation} \label{vani}
H^1\big(X\times X,
\big(\mathscr{G}\otimes\mathscr{I}_{\Delta}\big)(a,b)\big)=0.
\end{equation}
for $a,b\geqslant m_0$. Indeed, there is an isomorphism of
cohomology groups
\begin{align*}
H^0\big(X\times X, \mathscr{G}(a,b)&\big)\cong \\
&H^0\big(X,\mathscr{F}_1\otimes\mathscr{O}_X(aD)\big)\otimes
H^0\big(X,\mathscr{F}_2\otimes\mathscr{O}_X(bE)\big).
\end{align*}
By the projection formula,
$$
\big(\mathscr{G}\otimes\Delta_{\ast}\mathscr{O}_X\big)(a,b)\cong
\Delta_{\ast}\big(\Delta^{\ast}\mathscr{G}(a,b)\big).
$$
By definition of the diagonal morphism we have $p_i\Delta=id_X$, hence
$$
\Delta^{\ast}\big(p_1^{\ast}\mathscr{F}_1\otimes
p_2^{\ast}\mathscr{F}_2\otimes p_1^{\ast}(aD)\otimes
p_2^{\ast}(bE)\big)\cong\mathscr{F}_1 \otimes
\mathscr{F}_2\otimes\mathscr{O}_X(aD+bE).
$$
Therefore the cohomology group
$$
H^0\big(X\times X,
(\mathscr{G}\otimes\Delta_{\ast}\mathscr{O}_X)(a,b)\big)\cong
H^0\big(X\times X, \Delta_*\Delta^*\mathscr{G}(a,b)\big)
$$
is isomorphic to
$$
H^0\big(X,\Delta^{\ast}\mathscr{G}(a,b)\big) \cong
H^0\big(X,\mathscr{F}_1\otimes\mathscr{F}_2 \otimes
\mathscr{O}_X(aD+bE)\big)
$$
as desired.

Now we prove (\ref{vani}). To this end, we use the ample divisor
$aD\boxplus bE$ to construct a (possibly non-terminating)
resolution
\begin{equation} \label{ntresol}
 \cdots\longrightarrow \bigoplus\mathscr{O}_{X\times
X}(-p_1,-p_1)\longrightarrow \bigoplus\mathscr{O}_{X\times
X}(-p_0,-p_0)\longrightarrow\mathscr{G}\otimes\mathscr{I}_{\Delta}
\longrightarrow 0
\end{equation}
for suitable integers $0\leqslant p_0\leqslant p_1\leqslant\cdots$
where again $(a,b)$ means the twisting
$\otimes\mathscr{O}_{X\times X}(aD\boxplus bE)$. Set $d=\dim
X\times X$. By dimension shifting, to prove $(\ref{vani})$ it is
enough to produce an integer $m_0$ such that
\begin{equation*}
H^i\big(X\times X, \mathscr{O}_{X\times
X}(a-p_{i-1},b-p_{i-1})\big)=0
\end{equation*}
whenever $a,b\geqslant m_0$ and $i=0,1,\ldots,d-1$. In fact, we
have then
\begin{align*}
H^1\big(X\times X,
\big(\mathscr{G}\otimes\mathscr{I}_{\Delta}\big)(a,b)\big)
&\cong H^2\big(X\times X, \mathscr{K}_0(a,b)\big) \\
&\quad\vdots\\
&\cong H^d\big(X\times X, \mathscr{K}_{d-2}(a,b)\big)\\
&\cong H^{d+1}\big(X\times X, \mathscr{K}_{d-1}(a,b)\big)=0
\end{align*}
where $\mathscr{K}_i$ is the kernel of the morphism
$$
\bigoplus\mathscr{O}_{X\times X}(-p_i,-p_i)\rightarrow
\bigoplus\mathscr{O}_{X\times X}(-p_{i-1},-p_{i-1})
$$
for $i > 0$ and $\mathscr{K}_0$ is the kernel of
$$
\bigoplus\mathscr{O}_{X\times
X}(-p_0,-p_0)\rightarrow\mathscr{G}\otimes\mathscr{I}_{\Delta}
\rightarrow 0.
$$
The last group vanishes by dimension reason. The existence of the
required integer $m_0$ then follows from Serre's vanishing
theorem.

For general coherent sheaves $\mathscr{F}_j$, we can write
$\mathscr{F}_j$ as a quotient of a sheaf $\mathscr{E}_j$ which is
a finite direct sum of sheaves of the form $\mathscr{O}_X({q_i})$.
We consider the following exact sequence
$$
0\longrightarrow\mathscr{K} \longrightarrow\mathscr{E}_1
\otimes \mathscr{E}_2 \longrightarrow
\mathscr{F}_1\otimes\mathscr{F}_2\longrightarrow 0.
$$
Choose a positive integer $m_0$ such that
\begin{enumerate}
\item $H^1(X, \mathscr{K}\otimes\mathscr{O}_X(aD+bE))$ vanishes
for $a,b\geqslant m_0$, and

\item the multiplication map
\begin{align*}
H^0\big(X,\mathscr{E}_1 \otimes \mathscr{O}_X(aD)\big)\otimes
H^0\big(X,\mathscr{E}_2 \otimes& \mathscr{O}_X(bE)\big)\rightarrow \\
&H^0\big(X,\mathscr{E}_1 \otimes \mathscr{E}_2 \otimes
\mathscr{O}_X(aD+bE)\big)
\end{align*}
is surjective whenever $a,b\geqslant m_0$.
\end{enumerate}
Consider the commutative diagram \small
$$
\xymatrix{ {H^0\big(X,\mathscr{E}_1 \otimes
\mathscr{O}_X(aD)\big)\otimes H^0\big(X,\mathscr{E}_2 \otimes
\mathscr{O}_X(bE)\big)} \ar[r] \ar[d] &
{H^0\big(X,\mathscr{E}_1\otimes\mathscr{E}_2 \otimes
\mathscr{O}_X(aD+bE)\big)}
\ar[d]\\
{H^0\big(X,\mathscr{F}_1 \otimes \mathscr{O}_X(aD)\big)\otimes
H^0\big(X,\mathscr{F}_2 \otimes \mathscr{O}_X(bE)\big)} \ar[r] &
{H^0\big(X,\mathscr{F}_1 \otimes \mathscr{F}_2 \otimes
\mathscr{O}_X(aD+bE)\big)} }
$$ \normalsize
If $a,b\geqslant m_0$, the right vertical map is surjective by
$(1)$, and the upper horizontal map is surjective by $(2)$. So the
lower horizontal multiplication map is surjective for
$a,b\geqslant m_0$. This completes the proof.
\end{proof}

\begin{remark}
In the case $X$ being smooth the resolution $(\ref{ntresol})$ is
actually finite by the Hilbert Syzygy Theorem.
\end{remark}

\end{document}